\documentclass[11pt, letter]{amsart}
\usepackage[latin1]{inputenc}
\usepackage[english]{babel}
\usepackage{amsmath, amsfonts, amssymb, amsthm, mathtools, mathrsfs}
\usepackage{enumerate}
\usepackage[numbers]{natbib}
\usepackage{setspace}

\usepackage[arrow,matrix,curve]{xy}
\usepackage[pagebackref,hypertexnames=true]{hyperref}

\usepackage{tikz}
\usetikzlibrary{matrix}

\newtheorem{theorem}{Theorem}[section]
\newtheorem{lemma}[theorem]{Lemma}
\newtheorem{proposition}[theorem]{Proposition}
\newtheorem{corollary}[theorem]{Corollary}

\theoremstyle{definition}
\newtheorem{definition}[theorem]{Definition}

\newtheorem{remark}[theorem]{Remark}

\newcommand{\QQ}{M_{\mathbb{Q}}}

\newcommand{\ZZ}{\mathcal{Z}}
\newcommand{\OO}{\mathcal{O}}
\newcommand{\N}{\mathbb{N}}
\newcommand{\Z}{\mathbb{Z}}
\newcommand{\Q}{\mathbb{Q}}
\newcommand{\R}{\mathbb{R}}
\newcommand{\C}{\mathbb{C}}
\newcommand{\K}{\mathbb{K}}		

\newcommand{\Cuntz}[1]{\mathcal{O}_{#1}}
\newcommand{\Aut}[1]{{\rm Aut}(#1)}
\newcommand{\Endo}[1]{{\rm End}(#1)}
\newcommand{\uAut}[1]{{\rm Aut}_0(#1)}
\newcommand{\uEnd}[1]{{\rm End}_0(#1)}
\newcommand{\AutSt}[2]{{\rm Aut}_{#2}(#1)}
\newcommand{\EndSt}[2]{{\rm End}_{#2}(#1)}

\newcommand{\id}[1]{{\rm id}_{#1}}

\newcommand{\obj}{{\rm obj}}
\newcommand{\mor}{{\rm mor}}

\newcommand{\Proj}[1]{\mathcal{P}r(#1)}
\newcommand{\uProj}[1]{\mathcal{P}r_0(#1)}

\newcommand{\BunX}[1]{\mathcal{B}un_X(#1)}

\begin{document}
	
\title{A Dixmier-Douady theory for strongly self-absorbing
 $C^*$-algebras}
\author{Marius Dadarlat}\address{MD: Department of Mathematics, Purdue University, West Lafayette, IN 47907, USA}\email{mdd@math.purdue.edu}
\author{Ulrich Pennig}\address{UP: Mathematisches Institut, Westf\"alische Wilhelms-Universit\"at M\"unster, Ein\-stein\-stra\ss e 62, 48149 M\"unster, Germany}\email{u.pennig@uni-muenster.de}	
\begin{abstract}

We show that the  Dixmier-Douady theory of continuous fields of $C^*$-algebras with compact operators $\K$ as fibers extends significantly to a  more general theory
 of 
  fields with fibers $A\otimes \K$ where $A$ is a strongly self-absorbing C*-algebra.
  The classification of the corresponding locally trivial fields involves a  generalized cohomology theory   which is computable via the Atiyah-Hirzebruch spectral sequence.
  An important feature
 of the general theory is the appearance of characteristic  classes in higher dimensions.
 We also give a necessary and sufficient $K$-theoretical condition for local triviality  of these continuous fields over  spaces of finite covering dimension.
\end{abstract}
\thanks{M.D. was partially supported by NSF grant \#DMS--1101305}
\maketitle
\tableofcontents

\section{Introduction}
Continuous fields of C*-algebras are employed
	as versatile tools in several areas, including index and representation theory, the Novikov and the Baum-Connes conjectures \cite{Connes:bookNG}, (deformation) quantization \cite{Rieffel:Quantization, Landsman:book} and the study of the quantum Hall effect \cite{Bellissard:Hall-effect}. While continuous fields play the role of bundles in C*-algebra theory, the underlying bundle structure  is typically not locally trivial. Nevertheless, these bundles have sufficient continuity properties to allow for local propagation of interesting K-theory invariants along their fibers.
Continuous fields of C*-algebras with simple fibers occur naturally as  the class of C*-algebras with Hausdorff primitive spectrum.

 In a classic paper \cite{paper:DixmierDouady}, Dixmier and Douady  studied the continuous fields of $C^*$-algebras
 with fibers (stably) isomorphic to the compact operators $\K = \K(H)$
  ($H$ an infinite dimensional Hilbert space) over a paracompact base space $X$.
 In this article we develop a general theory
 of continuous
  fields with fibers $A\otimes \K$ where $A$ is a strongly self-absorbing C*-algebra. We show that the results of \cite{paper:DixmierDouady} fit naturally and admit significant generalizations in the new theory.
  The classification of these fields involves  suitable generalized cohomology theories.
  An important feature
 of the new theory is the appearance of characteristic  classes in higher dimensions.

 As a byproduct of our approach we find an operator algebra realization of the classic spectrum $BBU_\otimes$.
 Let us recall that for a compact connected metric space $X$ the invertible elements of the $K$-theory ring $K^0(X)$ is an abelian group $K^0(X)^\times$ whose elements
 are represented by classes of vector bundles of virtual rank $\pm 1$, corresponding to homotopy classes
 $[X,\Z/2\times  BU_\otimes]$. The group operation is induced by the tensor product of vector bundles.  Segal has shown that $BU_\otimes$ is in fact an infinite loop space and
 hence there is a cohomology theory $bu^*_\otimes(X)$ such that $K^0(X)^\times$ is just the $0$-group $bu^0_\otimes(X)$ of this theory  \cite{paper:SegalCatAndCoh}, but gave no geometric interpretation for the higher order groups. Our results lead to a geometric realization of the first group $bu^1_\otimes(X)$ as the  isomorphism classes of locally trivial bundles of C*-algebras with fiber
 the stabilized Cuntz algebra $\OO_\infty \otimes \K$
 where the group operation corresponds to the tensor product, see \cite{Dadarlat-Pennig:unit_spectra}.

Let us recall two central results of Dixmier and Douady from  \cite{paper:DixmierDouady}.
\begin{theorem}\label{thm:DD1} The isomorphism classes of locally trivial
fields over $X$ with fibers $\K$ form a group under the operation
of tensor product and this group is isomorphic to $\check{H}^3(X,\mathbb{Z})$.
\end{theorem}

\begin{theorem}\label{thm:DD2}If $X$ is finite dimensional,
then a separable continuous field $B$ over $X$ with fibers isomorphic to $\K$ is locally trivial if and only if it satisfies  Fell's condition, i.e.\ each point of
$ X$ has a closed neighborhood $V$ such that the restriction of $B$ to $V$
contains a projection of constant rank~$1$.
\end{theorem}
 The corresponding characteristic class $\delta(B) \in \check{H}^3(X,\Z)$ is now known as the Dixmier-Douady invariant. Most prominent among its applications is its appearance as twisting class in twisted $K$-theory \cite{paper:DonovanKaroubi, paper:RosenbergCT, paper:AtiyahSegal, paper:AtiyahSegal-2}, {which is the natural home for $D$-brane charges in string theory} \cite{paper:SegalString, paper:BouwknegtMathai}. A recent friendly introduction
 to the Dixmier-Douady theory can be found in  \cite{Schochet:dummy}.

The class of strongly self-absorbing C*-algebras, introduced by Toms and Winter \cite{paper:TomsWinter},
is closed under tensor products and  contains C*-algebras that are cornerstones of Elliott's classification program of simple nuclear C*-algebras:
the Cuntz algebras $\Cuntz{2}$ and $\Cuntz{\infty}$, the Jiang-Su algebra $\mathcal{Z}$, the canonical anticommutation relations algebra $M_{2^\infty}$ and in fact all UHF-algebras of infinite type. These are separable C*-algebras
singled out by a crucial property: there exists an isomorphism $A \to A \otimes A$, which is unitarily homotopic to the map $a \mapsto a \otimes 1_A$, \cite{paper:DadarlatKK1, paper:WinterZStable}. Using this property, which is equivalent to, but formally much stronger than  the original definition of \cite{paper:TomsWinter}, we prove that
\begin{itemize}
\item  $\Aut{A }$   is contractible in the point-norm topology.
\item $\Aut{A \otimes \K}$ is well-pointed and it has the homotopy type  of a CW-complex.
\item The classifying space $B\Aut{A \otimes \K}$ of locally trivial $C^*$-algebra bundles with fiber $A \otimes \K$ carries an $H$-space structure induced by the tensor product. Moreover, this tensor product multiplication is homotopy commutative up to all higher homotopies and therefore equips $B\Aut{A \otimes \K}$ with the structure of an infinite loop space by  results of Segal and May.
\end{itemize}
 These properties mirror entirely the corresponding properties  of
 $\Aut{\K}=PU(H)$ and $B\Aut{\K}=BPU(H)$ obtained by their identification with the Eilenberg-MacLane spaces $K(\Z,2)$ and respectively $K(\Z,3)$ which we implicitly reprove as they correspond to the case $A=\C$. Recall that if $X$ is paracompact Hausdorff, then $\check{H}^n(X,\Z)\cong [X,K(\Z,n)],$ \cite{paper:Huber}.

It is worth noting that while the obstructions to having a natural group structure on the isomorphism classes of locally trivial continuous fields with fiber  $A \otimes \K$ -- such as nontrivial Samelson products \cite[sec.6]{paper:Nistor_AF} -- do vanish in the strongly self-absorbing case, that is not necessarily true for general self-absorbing $C^*$-algebras, i.e. those with $A\otimes A\cong A$.
 This motivates yet again our choice of fibers.
In complete analogy with Theorem~\ref{thm:DD1} we have:
\vskip 8pt
\noindent{\bf Theorem A}. \emph{Let $X$ be a compact metrizable space and let $A$ be a strongly self-absorbing C*-algebra. The set $\BunX{A\otimes \K}$ of isomorphism classes of locally trivial
fields over $X$ with fiber $A\otimes \K$ becomes an abelian group under the operation
of tensor product. Moreover, this group is isomorphic to $E_A^1(X)$,
the first group of a generalized connective cohomology theory $E_{A}^*(X)$ defined
by the infinite loop space $B\Aut{A \otimes \K}$.}

We also show that the zero group $E_{A}^0(X)$ computes the homotopy classes
$[X,\Aut{A \otimes \K}]$ and it is isomorphic to the group of positive invertible elements of the abelian ring $K_0(C(X)\otimes A)$, denoted by $K_0(C(X)\otimes A)^{\times}_{+}$, for $A\neq\C$.
In particular, we fully compute the coefficients of $E_{A}^*(X)$, as they are given by the homotopy groups
 \[\pi_i(\Aut{A\otimes \K})=
\begin{cases}
		K_0(A)_{+}^{\times} & \text{if } i=0 \\
			K_i(A) & \text{if } i\geq 1\ .
	\end{cases}
\]
$K_0(A)$ has a natural ring structure with unit given by
the class of  $1_A$. $K_0(A)^{\times}$ denotes the group of multiplicative  elements of $K_0(A)$
and $K_0(A)_{+}^{\times}$ is its subgroup consisting of positive elements.

 The Atiyah-Hirzebruch spectral sequence then allows us to obtain classification results for locally trivial $A \otimes \K$-bundles over $X$. In the case of the universal UHF algebra $M_{\mathbb{Q}}$, bundles with fiber $M_{\mathbb{Q}} \otimes \K$  are essentially classified by the ordinary rational cohomology groups of odd degree of the underlying space:
 \[
\BunX{M_{\mathbb{Q}} \otimes \K}\cong E_{M_{\mathbb{Q}}}^{1}(X)  \cong H^1(X,\mathbb{Q}_+^\times)\oplus \bigoplus_{k \geq 1} H^{2k+1}(X,\mathbb{Q})
.\]
  A similar result holds for   bundles with fiber $\mathcal{O}_\infty \otimes M_{\mathbb{Q}} \otimes \K$,
see Corollary~\ref{cor:computationEA-1}.
It follows that if $A$ is any strongly self-absorbing C*-algebra  that satisfies the UCT, then there are
rational characteristic classes
\(\delta_k:\BunX{A \otimes \K} \to H^{2k+1}(X,\Q)\) such that  $\delta_k(B_1\otimes B_2)=\delta_k(B_1)+\delta_k(B_2)$.

An unexpected consequence of our results is that for any strongly self-absorbing C*-algebra $A$, if two bundles $B_1,B_2\in \BunX{A\otimes \K}$ become isomorphic after tensoring with $\OO_\infty$,  then they must be isomorphic in the first place, see Corollary~\ref{cor:103}.

Our result concerning local triviality is the following generalization of Theorem~\ref{thm:DD2}
 which involves a $K$-theoretic interpretation of Fell's condition.
\vskip 6pt
\noindent{\bf Theorem B}. \emph{Let $X$ be a locally compact metrizable space of finite covering dimension and   let $A$ be a strongly self-absorbing C*-algebra. A separable continuous field  $B$ over $X$ with fibers abstractly isomorphic to $A\otimes \K$
 is locally trivial if and only if for each point
$x\in X$, there exist a closed neighborhood $V$ of $x$ and a projection $p\in B(V)$ such that
$[p(v)]\in K_0(B(v))^{\times}$  for all $v\in V$.}
\vskip 4pt

{A notable consequence of Theorem B is that
  any separable continuous field of C*-algebras over $X$ with all fibers abstractly isomorphic to $M_{\Q}\otimes \K$ is locally trivial and therefore, by Theorem A, it is determined up to isomorphism by its class in $ E_{M_{\mathbb{Q}}}^{1}(X)  \cong H^1(X,\mathbb{Q}_+^\times)\oplus \bigoplus_{k \geq 1} H^{2k+1}(X,\mathbb{Q})
.$}

 The condition that $X$ is finite dimensional is essential in Theorem B, as shown by examples  constructed in
 \cite{paper:DixmierDouady} for $A=\mathbb{C}$, \cite{paper.Hirshberg.Rordam.Winter} for $A=M_{\mathbb{Q}}$   and \cite{Dad:fiberwiseKK} for $A=\OO_2$.

Let us recall that a C*-algebra isomorphic  to the compact operators on a separable (possibly finite
dimensional) Hilbert space is called an  elementary C*-algebra.
Dixmier and Douady gave two other results concerning continuous fields of elementary C*-algebras:

(i) If $B$ is a continuous field
of elementary C*-algebras that satisfies Fell's condition,  then
$B\otimes \K$ is locally trivial.

(ii) The class $\delta(B) \in \check{H}^3(X,\Z)$ can be defined for any continuous field
of elementary C*-algebras that satisfies Fell's condition. Moreover $B$ is isomorphic to the
compact operators of a continuous field of Hilbert spaces if and only if $\delta(B)=0$.

We extend (i)  and the first part of (ii) to  general strongly
self-absorbing C*-algebras, but  we must require finite dimensionality for either the fiber or  the base space
in order to obtain a perfect analogy with these results, see Corollaries~\ref{cor:D} and~\ref{cor:char-class}. These  restrictions are necessary.  Indeed, while any  unital separable continuous field of C*-algebras with fiber $\mathbb{C}$ over $X$ is locally trivial (in fact isomorphic to $C_0(X)$), automatic local triviality
fails if $\mathbb{C}$ is replaced by  strongly self-absorbing C*-algebras such as $M_{\mathbb{Q}}$ and $\OO_2$, see \cite{paper.Hirshberg.Rordam.Winter}  and \cite{Dad:fiberwiseKK}.

This fact  also explains why the second part of (ii) is specific to fields of elementary C*-algebras.
Our set-up allows us to associate rational characteristic classes to any continuous fields (satisfying a weak Fell's condition)
whose fibers are Morita equivalent to  strongly self-absorbing C*-algebras which are not necessarily mutually
isomorphic. Such fields are typically very far from being locally trivial.
 We refer the reader to Section~\ref{section:DD} for further discussion.

The homotopy equivalence $\Aut{A \otimes \K}\simeq K_0(A)_{+}^{\times} \times BU(A)$ (see Corollary~\ref{cor:AutAK_BUA}) suggests that the generalized cohomology theory associated to $\Aut{A \otimes \K}$ is very closely related to the unit spectrum $GL_1(KU^A)$ of topological $K$-theory with coefficients in the group $K_0(A)$. This is again parallel to the Dixmier-Douady theory, where we have $\Aut{\K} = PU(H) \simeq BU(1) \subset GL_1(KU)$. We will make this connection precise in \cite{Dadarlat-Pennig:unit_spectra}.
Let us  just mention here that the homotopy equivalence
$\Aut{\ZZ \otimes \K}\simeq BU$ deloops to a homotopy equivalence
  $B\Aut{\ZZ \otimes \K}\simeq B(BU_{\otimes}).$ This unveils a very natural
  operator algebra realization of the classic $\Omega$-spectrum $B(BU_{\otimes})$ introduced by Segal \cite{paper:SegalCatAndCoh}.

  The authors are grateful to Johannes Ebert, Peter May and Jim McClure for a number of useful discussions.

\section{The topology of $\Aut{A \otimes \K}$ for strongly self-absorbing algebras }

 The automorphism group $\Aut{B}$ of a separable C*-algebra $B$, equipped with the point-norm topology, is a separable and  metrizable topological group.
In particular its topology is compactly generated.
We are going to show in this section that if  $A$ is a strongly self-absorbing $C^*$-algebra, then  $\Aut{A \otimes \K}$ is well-pointed and has the homotopy type of a $CW$-complex.
This will enable us to apply the standard techniques of algebraic topology, in particular when it comes to dealing with its classifying space. We denote by $\simeq$ the relation of homotopy equivalence.
\subsection{Strongly self-absorbing C*-algebras}
Let us recall from \cite{paper:TomsWinter} that a C*-algebra $A$ is strongly self-absorbing if it
is separable, unital and there exists a $*$-isomorphism $\psi:A \to A\otimes A$ such that
$\psi$ is approximately unitarily equivalent to the map $l:A \to A\otimes A$, $l(a)=a\otimes 1_A$.
It follows from \cite{paper:DadarlatKK1} and \cite{paper:WinterZStable} that  $\psi$ and $l$ must be in fact unitarily homotopy equivalent, see Theorem~\ref{thm:cancellation}(b). Note that, unlike \cite{paper:TomsWinter},
 we don't  exclude the complex numbers $\C$ from the class of strongly self-absorbing C*-algebras.
 For future reference, we collect under one roof an important  series of results due to several authors.
\begin{theorem}\label{thm:cancellation}
 A strongly self-absorbing C*-algebra $A$ has the following properties:
\end{theorem}
\begin{itemize}
\item[(a)] $A$ is simple, nuclear and is either stably finite or purely infinite; if it is stably finite, then it admits
a unique trace, see \cite{paper:TomsWinter} and references therein.

\item[(b)] Let $B$ be a unital separable C*-algebra. For any two unital $*$-homomorphisms $\alpha,\beta:A\to B\otimes A$ there is a continuous path of unitaries
$(u_t)_{t\in [0,1)}$ in $B\otimes A$ such that $u_0=1$ and  $\lim_{t \to 1}\|u_t\alpha(a)u_t^*-\beta(a)\|=0$ for all $a\in A$.
This property was proved in \cite[Thm.2.2]{paper:DadarlatKK1} under the assumption that $A$ is $K_1$-injective.
 Winter \cite{paper:WinterZStable} has shown that any infinite dimensional strongly self-absorbing C*-algebra $A$ is $\mathcal{Z}$-stable,
i.e. $A \otimes \mathcal{Z}\cong A$, and hence $A$ is $K_1$-injective  by a result of R{\o}rdam~\cite{paper:Rordam-ZStable}.

\item[(c)] Any unital  $\ZZ$-stable C*-algebra has cancellation of full projections by a result of Jiang
\cite[Thm.1]{paper:Jiang-nonstable}.  In particular,
if $B$ is a separable unital C*-algebra and $A\neq \C$, then $B\otimes A$
is isomorphic to $B\otimes A\otimes \mathcal{Z}$  and hence it has cancellation of full projections.

\item[(d)] If $B$ is a unital $\ZZ$-stable C*-algebra, then $\pi_0(U(B))\cong K_1(B)$, by \cite[Thm.2]{paper:Jiang-nonstable}.

\item[(e)]
If $A$ satisfies the Universal Coefficient Theorem (UCT) in KK-theory, then $K_1(A)=0$ by \cite{paper:TomsWinter}.
If  in addition  $A$ is purely infinite, then  $A$ is isomorphic to either $\Cuntz{2}$ or $\Cuntz{\infty}$ or a tensor product of $\Cuntz{\infty}$ with a UHF-algebra of infinite type \cite[Cor.5.2]{paper:TomsWinter}.
\end{itemize}
\noindent\textbf{Notation.} For  C*-algebras $A,B$ we denote by $\mathrm{Hom}(A,B)$ the space
of full $*$-homomorphisms from $A$ to $B$
and by $\mathrm{End}(A)$ the space of full $*$-endomorphisms of $A$.
Recall that a $*$-homomorphism $\varphi:A\to B$ is full if for any nonzero element $a\in A$, the closed ideal generated by $\varphi(a)$
 is equal to $B$. If $A$ is a unital C*-algebra, we denote by $K_0(A)_{+}$ the subsemigroup of $K_0(A)$
 consisting of classes $[p]$ of full projections $p\in A\otimes \K$.

\subsection{Contractibility of $\Aut{A}$}
While it is known from \cite[Cor.3.1]{paper:DadarlatKK1} that $\Aut{A}$ is weakly contractible in the point norm-topology, we can strengthen this result by combining it with the idea of half-flips from \cite{paper:TomsWinter}.

Let $B$ be a separable $C^*$-algebra and let $e\in B$ be a projection.
Consider the following spaces of $*$-endomorphisms of $B$ endowed with the point-norm topology.
\[\mathrm{End}_{e}(B)=\{\alpha \in \mathrm{End}(B) \colon\, \alpha(e)=e\}, \quad \mathrm{Aut}_{e}(B)=\{\alpha \in \mathrm{Aut}(B) \colon\, \alpha(e)=e\}.\]
Let  $l,r:B \to B\otimes B$ (minimal tensor product) be defined by $l(b) = b \otimes e$ and $r(b) = e \otimes b$.
\begin{lemma}\label{lemma:Aut(B,e)contractible}  Suppose that there is a continuous map $\psi:[0,1]\to \mathrm{Hom}(B, B\otimes B)$ such that $\psi(0)=l$, $\psi(1)=r$, $\psi(t)(e)=e\otimes e$ and $\psi(t)$ is a $*$-isomorphism for all $t\in (0,1)$. Then $\mathrm{Aut}_{e}(B)$ and $\mathrm{End}_{e}(B)$ are contractible spaces.
\end{lemma}
\begin{proof} First we deal with $\mathrm{Aut}_{e}(B)$.
 Consider $H \colon I \times \mathrm{Aut}_{e}(B) \to \mathrm{Aut}_{e}(B)$ defined by
\begin{equation}\label{eqn:basic-homotopy}
 H(t,\alpha) =
	\begin{cases}
		\alpha & \text{for } t = 0 \\
		 \psi(t)^{-1} \circ (\alpha \otimes \id{B}) \circ \psi(t) & \text{for } 0 < t <1\ ,\\
		\id{B} & \text{for } t = 1\ .
	\end{cases}
\end{equation}
Note that $H(t,\alpha)(e)=e$ since $\psi(t)(e)=e\otimes e$.
Observe that $(\alpha \otimes \id{B}) \circ l = l \circ \alpha$. It is straightforward to verify the continuity of $H$
at points $(\alpha,t)$ with $t\ne 0$ and $t\neq 1$.
Let $b \in B$, let $t_n \in (0,1)$ be a net converging to $0$ and let $\alpha_i \in \mathrm{Aut}_{e}(B)$ be a net converging to $\alpha \in \mathrm{Aut}_{e}(B)$. The estimate,
\begin{align*}
	& \lVert (\psi(t_n)^{-1} \circ (\alpha_i \otimes \id{B}) \circ \psi(t_n))(b) - \alpha(b) \rVert
	=  \lVert (\alpha_i \otimes \id{B}) \circ \psi(t_n)(b) - \psi(t_n) \circ \alpha(b) \rVert \\
	\leq \quad & \lVert (\alpha_i \otimes \id{B}) \circ \psi(t_n)(b) - (\alpha_i \otimes \id{B}) \circ l(b) \rVert + \lVert (\alpha_i \otimes \id{B}) \circ l(b) - (\alpha \otimes \id{B}) \circ l(b) \rVert  \\
	+ \quad & \lVert (\alpha \otimes \id{B}) \circ l(b) - \psi(t_n) \circ \alpha(b) \rVert \\
	\leq \quad & \lVert \psi(t_n)(b) - l(b) \rVert + \lVert \alpha_i(b) - \alpha(b) \rVert + \lVert (\alpha \otimes \id{B}) \circ l(b) - \psi(t_n) \circ \alpha(b) \rVert
\end{align*}
implies the  continuity of $H$ at $(\alpha,0)$. An analogous argument using $(\alpha \otimes \id{B}) \circ r = r$
shows continuity at $(\alpha,1)$. We also have $H(t,\id{B})=\id{B}$ for all $t\in[0,1]$.
Thus, $H$ provides a (strong) deformation retraction of $\mathrm{Aut}_{e}(B)$ to $\id{B}$.
The argument for the contractibility of $\mathrm{End}_{e}(B)$ is entirely similar. One observes that
 equation~\eqref{eqn:basic-homotopy} also defines a map $H \colon I \times \mathrm{End}_{e}(B) \to \mathrm{End}_{e}(B)$ which gives a deformation retraction of $\mathrm{End}_{e}(B)$ to $\id{B}$..
\end{proof}

\begin{theorem} \label{thm:AutAcontractible}
Let $A$ be a  strongly self-absorbing $C^*$-algebra. Then $\Aut{A}$ and  $\mathrm{End}_{1_A}(A)$ are contractible spaces.
\end{theorem}

\begin{proof} \label{pf:AutAcontractible}
Let $l, r \colon A \to A \otimes A$ be the maps $l(a) = a \otimes 1_A$ and $r(a) = 1_A \otimes a$. Fix an isomorphism $\psi \colon A \to A \otimes A$. It follows from  Theorem~\ref{thm:cancellation}(b) that there exists a continuous path of unitaries $u\colon (0, 1] \to U(A \otimes A)$ with $u(1) = 1_{A \otimes A}$ such that
\[
	\lim_{t \to 0}\lVert u(t)\,\psi(a)\,u(t)^* - l(a) \rVert = 0\ .
\]
Define $\psi_l \colon (0,1] \to {\rm Iso}(A, A \otimes A)$ by $\psi_l(t) = {\rm Ad}_{u(t)} \circ \psi$. Likewise there is a continuous path of unitaries $v \colon [0,1) \to U(A \otimes A)$ with $v(0) = 1_{A \otimes A}$ and such that
\[
	\lim_{t \to 1}\lVert v(t)\,\psi(a)\,v(t)^* - r(a) \rVert = 0\ .
\]
Define $\psi_r \colon [0,1) \to {\rm Iso}(A, A \otimes A)$ by $\psi_r(t) = {\rm Ad}_{v(t)} \circ \psi$. By juxtaposing the paths $\psi_l$ and $\psi_r$ we obtain a homotopy from $l$ to $r$ which satisfies the assumptions of Lemma~\ref{lemma:Aut(B,e)contractible} with $e=1_A$. It follows  that $\Aut{A}$ and  $\mathrm{End}_{1_A}(A)$ are contractible spaces.\end{proof}
The following is a minor variation of a result of Blackadar \cite[p.57]{book:BlackadarK-theory} and Herman and Rosenberg \cite{paper:Herman-Rosenberg}.
\begin{lemma}\label{lem:AF-paths}
Let $A$ and $B$  be separable AF-algebras and let $e\in A$ be a projection.
Suppose that $\varphi,\psi:A\to B$ are two $*$-homomorphisms such that $\varphi(e)=\psi(e)$
and $\varphi_*=\psi_*:K_0(A)\to K_0(B)$.
Then there is a continuous map $u:[0,1)\to U(B^{+})$ with $u(0)=1$, $[u(t),\psi(e)]=0$ for all $t\in [0,1)$ and such that
\(\lim_{t\to 1} \|u(t)\psi(a)u(t)^*-\varphi(a)\|=0\) for all $a\in A$.\end{lemma}
\begin{proof} If $B$ is a nonunital $C^*$-algebra, we regard $B$ as a C*-subalgebra of its unitization  $B^{+}$.
Write $A$ as the closure of an increasing union of finite dimensional
C*-subalgebras $A_n \subset A_{n+1}$ with $A_0=\mathbb{C}e$.
Since $\phi_*=\psi_*$, for each $n\geq 0$ we find a unitary $u_n\in U(B^{+})$
such that $u_n\psi(x)u_n^*=\varphi(x)$ for all $x\in A_n$ and $u_0=1$. Observe that $w_n=u_{n+1}^*u_n$
is a unitary in the commutant $C_n$ of $\psi(A_n)$ in $B^{+}$. This commutant is known to be an AF-algebra, see \cite[Lemma 3.1]{paper:Herman-Rosenberg}. Therefore there is a continuous
path of unitaries $t\mapsto W_n(t)\in U(C_n)$, $t\in [n,n+1]$,  such that $W_n(n)=w_n$ and $W_n(n+1)=1$.
Define a continuous map $u:[0,\infty)\to U(B)$ by $u(t)=u_{n+1}W_n(t)$, $t\in [n,n+1]$. One verifies immediately
that $[u(t),\psi(e)]=0$ for all $t$ and that $u(t)\psi(x)u(t)^*=\varphi(x)$ for all $x\in A_n$ and
$t\in [n,n+1]$. It follows that \(\lim_{t\to \infty} \|u(t)\psi(a)u(t)^*-\varphi(a)\|=0\)
for all $a\in A$.
\end{proof}

\begin{theorem}\label{thm:stabContractible}
Let $A$ be a  strongly self-absorbing $C^*$-algebra and let $e \in \K$ be a rank-$1$ projection. Then
the stabilizer group \(
	\AutSt{A \otimes \K}{1 \otimes e} \)
and the space $\mathrm{End}_{1\otimes e}(A\otimes \K)$ are contractible.
\end{theorem}

\begin{proof}
We shall use the following consequence of Lemma~\ref{lem:AF-paths}.
Let $\varphi_0,\varphi_1:\K\to \K \otimes \K$
be two $*$-homomorphisms such that $\varphi_0(e)=\varphi_1(e)=e\otimes e$.
Fix a $*$-isomorphism $\psi_{1/2}:\K\to \K \otimes \K$ with $\psi_{1/2}(e)=e\otimes e$. By applying Lemma~\ref{lem:AF-paths} to both pairs $(\varphi_i,\psi_{1/2})$, $i=0,1$, we
find a continuous map $\psi:[0,1]\to \mathrm{Hom}(\K, \K\otimes \K)$ such that $\psi(0)=\varphi_0$, $\psi(1)=\varphi_1$, $\psi(t)(e)=e\otimes e$ and $\psi(t)$ is a $*$-isomorphism for all $t\in (0,1)$.

We proceed in much the same way as  the proof of Theorem \ref{thm:AutAcontractible}, by applying
Lemma~\ref{lemma:Aut(B,e)contractible}.
Let $l, r \colon A \to A \otimes A$ be defined by $l(a) = a \otimes 1_A$ and $r(a) = 1_A \otimes a$. We have seen in the proof of Theorem \ref{thm:AutAcontractible}
that there is a continuous map $\psi:[0,1]\to \mathrm{Hom}(A, A\otimes A)$ such that $\psi(0)=l$, $\psi(1)=r$, and $\psi(t)$ is a $*$-isomorphism for all $t\in (0,1)$.

Let $l_{\K}, r_{\K} \colon \K \to \K \otimes \K$ be given by $l_{\K}(x) = x \otimes e$, $r_{\K}(x) = e \otimes x$. Using the remark from the beginning of the proof, we find a continuous map $\psi_\K:[0,1]\to \mathrm{Hom}(\K, \K\otimes \K)$ such that $\psi_\K(0)=l_\K$, $\psi_\K(1)=r_\K$, $\psi_\K(t)(e)=e\otimes e$ and $\psi_\K(t)$ is a $*$-isomorphism for all $t\in (0,1)$.

Let $A_{\K} = A \otimes \K$ and consider the $*$-homomorphisms $\hat{l},\ \hat{r} \colon A_{\K} \to A_{\K} \otimes A_{\K}$ with
\[
	\hat{l} = \sigma \circ (l \otimes l_{\K}) \quad \text{and} \quad \hat{r} = \sigma \circ (r \otimes r_{\K})\ ,
\]
where $\sigma \colon A \otimes (A \otimes \K) \otimes \K \to (A \otimes \K) \otimes (A \otimes \K)$ interchanges the second and third tensor factor. Note that
$\hat{l}(a\otimes x)=(a\otimes x) \otimes (1_A \otimes e)$ and $\hat{r}(a\otimes x)=(1_A\otimes e) \otimes (a\otimes x) $ for $a\in A$ and $x\in \K$.
Define $\hat{\psi}:[0,1]\to \mathrm{Hom}(A_{\K},A_{\K}\otimes A_{\K})$ by
$\hat{\psi}=\sigma\circ (\psi\otimes \psi_\K)$. Then $\hat{\psi}(0)=\hat{l}$,
$\hat{\psi}(1)=\hat{r}$, $\hat{\psi}(t)(1_A\otimes e)=(1_A\otimes e)\otimes (1_A\otimes e)$ and $\hat{\psi}(t)$ is an isomorphism for all $t\in (0,1)$. It follows by
Lemma~\ref{lemma:Aut(B,e)contractible} that $\AutSt{A \otimes \K}{1 \otimes e}$ and $\mathrm{End}_{1\otimes e}(A\otimes \K)$ are contractible.
\end{proof}
\begin{remark} Taking $A=\C$, Thm.~\ref{thm:stabContractible} reproves the contractibility of $U(H)$
in the strong topology.
\end{remark}

\subsection{The homotopy type of $\uAut{A \otimes \K}$ }
For a  $C^*$-algebra $B$ we denote by $\uAut{B}$ and $\uEnd{B}$ the path-connected component of the identity.
We have seen in Theorem~\ref{thm:AutAcontractible} that for a strongly self-absorbing $C^*$-algebra $A$ the space $\Aut{A}$ is contractible. In particular, it has the homotopy type of a CW-complex. In this section, we will extend the latter statement to the space $\uAut{A \otimes \K}$, which is no longer contractible, but has a very interesting homotopy type. We start by considering the subspace of projections in $A \otimes \K$, denoted by $\Proj{A \otimes \K}$.

\begin{lemma}\label{lem:ProjCW}
Let $B$ be a $C^*$-algebra. The space $\Proj{B}$ has the homotopy type of a CW-complex.
\end{lemma}

\begin{proof}
Let $B_{sa}$ be the real Banach space of self-adjoint elements in $B$. Consider the subset
$U$ of $B_{sa}$ consisting of all elements which do not have $1/2$ in the spectrum.
Since invertibility is an open condition, $U$ is an open subset of $B_{sa}$ and therefore has the homotopy type of a CW-complex by \cite[Cor.5.5, p.134]{book:LundellWeingram}. Since $\sigma(p) \subset \{0,1\}$ for any projection $p \in B$, we have $\Proj{B} \subset U$. Let $f$ be  the characteristic function
of the interval $(\frac{1}{2},\infty)$.
By  functional calculus, $f$ induces a continuous map
\(
	 U \to \Proj{B}, \,a \mapsto f(a)\ ,
\)
which restricts to the identity on $\Proj{B}$. Thus, $\Proj{B}$ is dominated by a space having the homotopy type of a CW-complex. By \cite[Cor.3.9, p.127]{book:LundellWeingram} it is homotopy equivalent to a CW-complex itself.
\end{proof}

Let $e$ be a rank-$1$ projection in $\K$. We define $\uProj{A \otimes \K}$ to be the connected component of $1 \otimes e \in \Proj{A \otimes \K}$. It does not depend on the choice of $e$ as long as the rank of $e$ is equal to $1$.

\begin{lemma}\label{lem:AutAKfibration}
Let $A$ be a unital $C^*$-algebra and let $e \in \K$ be a rank-$1$ projection. Then {the maps}
\(
	\uAut{A \otimes \K} \to \uProj{A \otimes \K}
\)
and
\(
	\uEnd{A \otimes \K} \to \uProj{A \otimes \K}
\)
{which send $\alpha$ to $\alpha(1 \otimes e)$ are locally trivial fiber bundles over a paracompact base space and therefore Hurewicz fibrations.}
\end{lemma}

\begin{proof}\label{pf:AutAKfibration}
This is a particular case of a more general result, which we will prove for $\uEnd{A \otimes \K}$. The proof for the sequence of automorphism groups is entirely analogous. Let $B$ be a C*-algebra, let $q\in \Proj{B}$ and let
$\uProj{B}$ be the path-component of $q$. Then
\(
	\pi: \uEnd{B} \to \uProj{B}
\), $\pi(\alpha)=\alpha(q)$
 is in fact a locally trivial bundle with fiber $\EndSt{B}{q}$.
 The map $\pi $  is well-defined. Indeed, if $\alpha$ is homotopic to $\id{B}$, then the projection $\alpha(q)$ is connected to $q$ by a continuous path in $\Proj{B}$.

Let $U_0(B^+)$ denote  the path-component of $1$ in the unitary group of the unitization of $B$.
 Thus, for $u \in U_0(B^+)$ we have ${\rm Ad}_u \in \uAut{B} \subseteq \uEnd{B}$. By definition any $p \in \uProj{B}$ is homotopic to $q$. Therefore $p$ and $q$ are also unitarily equivalent via a unitary $u \in U_0(B^+)$.
 Since\ $\pi({\rm Ad}_u) = p$ it follows that $\pi$ is surjective. Let $p_0 \in \uProj{B}$ and let $U$ be its the open neighborhood given by
\(
	U = \{ p \in \uProj{B}\ |\ \lVert p - p_0 \rVert < 1 \}.
\)
If $p\in U$, then $x_p = p_0p + (1-p_0)(1-p)$   is an invertible element of $B^+.$ It follows that $u_p = x_p(x_p^*x_p)^{-\frac{1}{2}}$ is a unitary in $U_0(B^+)$ and the map $p \mapsto u_p$ is continuous with respect to the norm topologies \cite[Prop.II.3.3.4]{book:BlackadarOpAlg}. Choose a unitary $v\in U_0(B^+)$ such that $p_0 = vqv^*$. Then
\(
	\sigma_{p_0} \colon U \to \uAut{B}, \, p \mapsto {\rm Ad}_{u_p^*v}
\)
is a continuous section of $\pi$ over $U$ and $\kappa_U \colon U \times \EndSt{B}{q} \to \uEnd{B}$ defined by $\kappa_U(x,\beta) = \sigma_{p_0}(x) \circ \beta$ is a local trivialization with inverse $\tau_U \colon \uEnd{B} \to U \times \EndSt{B}{q}$ given by $\tau_U(\beta) = (\beta(q), \sigma_{p_0}(\beta(q))^{-1} \circ \beta)$. This completes the proof.
\end{proof}

\begin{corollary}\label{cor:CWcomplex}
Let $A$ be a  strongly self-absorbing $C^*$-algebra. Then the spaces $\uAut{A \otimes \K}$ and $\uEnd{A \otimes \K}$ both have the homotopy type of a CW-complex; they are homotopy equivalent to $\uProj{A \otimes \K} \simeq BU(A)$.
\end{corollary}

\begin{proof}
By Lemma \ref{lem:AutAKfibration} and Theorem \ref{thm:stabContractible} the spaces $\uAut{A \otimes \K}$ and $\uEnd{A \otimes \K}$ are total spaces of fibrations where both base and fiber have the homotopy type of a CW-complex. Now the statement follows from \cite[Thm.2]{paper:Schoen}, Theorem \ref{thm:stabContractible}, except that it remains to argue that $\uProj{A \otimes \K} \simeq BU(A)$. This is certainly known.
The group $U(M(A \otimes \K))$ acts continuously and transitively on $\uProj{A \otimes \K}$ via $u \mapsto u(1 \otimes e)u^*$ with stabilizer
$U(A) \times U(M(A \otimes \K))$.
By the contractibility of $U(M(A \otimes \K))$ \cite{paper:CuntzHigson},
\(
	U(A)\to  U(M(A \otimes \K)) / 1_A\times U(M(A \otimes \K))\to \uProj{A \otimes \K}
\)
 is the \emph{universal} principal $U(A)$-bundle. One  uses  the map $p \mapsto u_p$ constructed in the proof of Lemma~\ref{lem:AutAKfibration} in order to verify local triviality. Thus $\uProj{A \otimes \K}$ is a model for $BU(A)$.
\end{proof}

\subsection{The homotopy type of $\Aut{A \otimes \K}$}
In this section we compute the homotopy classes $[X,\mathrm{End}(A\otimes \K)]$ and
$[X,\mathrm{Aut}(A\otimes \K)]$ in the case of  a  strongly self-absorbing C*-algebra $A$
and a compact metrizable space $X$, see Theorem~\ref{thm:MapsEndAK}. A similar topic was studied
for Kirchberg algebras in \cite{paper:DadarlatPiAutAK}. Throughout this subsection $e$ is a rank-$1$ projection in $\K$.
Given a unital ring $R$, we denote by $R^{\times}$ the group of units in $R$.  It is easily seen that $K_0(C(X) \otimes A)$  carries a ring structure with multiplication induced by an isomorphism $\psi: A\otimes \K\to A \otimes \K  \otimes A \otimes \K$ which maps $1_A \otimes e $ to $1_A \otimes e\otimes 1_A \otimes e$. This structure does not depend on the choice $\psi$ by Theorem~\ref{thm:stabContractible}. Let $\Endo{A \otimes \K}^{\times} = \{ \beta \in \Endo{A \otimes \K}\ |\ \beta(1 \otimes e) \text{ invertible in } K_0(A)\}$.
We identify the space of continuous maps from $X$ to $\Endo{A \otimes \K}$ with $\mathrm{Hom}(A\otimes \K, C(X)\otimes A \otimes \K)$ and with $\mathrm{End}_{C(X)}(C(X)\otimes A \otimes \K)$.
Similarly, we will identify the space of continuous maps from $X$ to $\Aut{A \otimes \K}$ with  $\mathrm{Aut}_{C(X)}(C(X)\otimes A \otimes \K)$.

\begin{lemma}\label{lem:homotopy-of-endos}
 Let $A$ and $B$ be unital separable C*-algebras. Suppose
that $p\in B\otimes \K$ is a  full projection such that there is a unital $*$-homomorphism $\theta:A\to p(B\otimes \K)p$.
  Then there is a $*$-homomorphism
 $\varphi:A\otimes \K \to B \otimes \K$ such that $\varphi(1\otimes e)=p$.
  If $\theta$ is an isomorphism, then
 we can choose $\varphi$ to be an isomorphism.
\end{lemma}
\begin{proof}
We denote by $\sim$ Murray-von Neumann equivalence of projections. Let us recall that if
$q,r\in B\otimes \K$ are projections, then $q\sim r$ if and only if there is $u\in U(M(B\otimes \K))$ such that
$uqu^*=r$, \cite[Lemma 1.10]{paper:Mingo}.
Since $p$ is a full projection in $B\otimes \K$,
by \cite{paperL.G.Brown.Stable.Isom}, there is $v\in M(B\otimes\K\otimes \K)$ such
that $v^*v=p\otimes I$ and $vv^*=1\otimes I\otimes I$. Then $\gamma: p(B\otimes \K)p\otimes \K \to B\otimes \K\otimes \K$, $\gamma(a)=vav^*$, is an isomorphism with the property that $\gamma(p\otimes e)=v(p\otimes e)v^*\sim (p\otimes e)(v^*v)(p\otimes e)=p\otimes e$.
The map $\K\to \K\otimes \K$, $x\mapsto x\otimes e$ is homotopic to a $*$-isomorphism
as observed in the proof of Theorem~\ref{thm:stabContractible}. It follows that the map
$B\otimes \K\to B\otimes\K\otimes\K $, $b\otimes x\mapsto b\otimes x\otimes e$ is also homotopic
to a $*$-isomorphism $\mu$. Note that $\mu(p)\sim p\otimes e\sim \gamma(p\otimes e)$.
Thus, after conjugating $\mu$
by a unitary in $M(B\otimes \K)$ we may arrange that $\mu(p)= \gamma(p\otimes e)$.
It follows that $\varphi=\mu^{-1}\circ\gamma\circ (\theta\otimes \mathrm{id}_{\K})\in \mathrm{Hom}(A\otimes \K,B\otimes \K)$ has the property that $\varphi(1\otimes e)=p$.  Finally note that
if $\theta$ is an isomorphism then so is $\varphi$.
\end{proof}
\begin{corollary}[Kodaka, \cite{paper:Kodaka}]\label{Kodaka}
 Let $A$ be a separable unital C*-algebra
and let $p\in A\otimes \K$ be a full projection.
Then $p(A\otimes \K)p\cong A$
if and only if there is $\alpha\in \Aut{A\otimes \K}$ such that $\alpha(1\otimes e)=p$.
\end{corollary}
\begin{proposition}\label{prop:homotopy-of-endoss}
 Let $A$ be a  strongly self-absorbing C*-algebra and let $B$ be a separable unital
C*-algebra such that $B\cong B\otimes A$. Let $\varphi,\psi:A\otimes \K \to B\otimes \K$ be two full $*$-homomorphisms. Suppose that
$ [\varphi(1_A \otimes e)]= [\psi(1_A \otimes e)]
\in K_0(B)$. Then (i) $\varphi$ is homotopic to $\psi$ and
(ii) $\varphi$ is approximately unitarily equivalent to
 $\psi$, written $\varphi \approx_{u} \psi$.
\end{proposition}
\begin{proof}
 (i) For C*-algebras $A$, $B$ we denote by $[A,B]_{\sharp}$ the homotopy classes of full $*$-homomorphisms
 $\varphi:A\to B$.
 The inclusion $A \cong A\otimes e \hookrightarrow A\otimes \K$ induces a restriction map
$\rho: [A\otimes \K,B\otimes \K]_{\sharp}\to [A,B\otimes \K]_{\sharp}$.
Thomsen showed that $\rho$ is bijective,
see \cite[Lemma 1.4]{Thomsen:homotopy-duke}.
Since the map $[\varphi]\mapsto [\varphi(1\otimes e)]$ factors through $\rho$, it suffices to show that the map
$ [A,B\otimes \K]_{\sharp}\to K_0(B)$, $\varphi\mapsto [\varphi(1)]$ is injective.
Let $\varphi,\psi: A\to B\otimes \K$ be two full $*$-homomorphisms.
Suppose that $[\varphi(1)]=[\psi(1)]$.  Since $B$   has cancellation of full projections by Theorem~\ref{thm:cancellation}(c), after conjugation by a unitary in the contractible group $U(M(B\otimes \K))$, we may assume that
$\varphi(1)=\psi(1)=p\in\Proj{B\otimes \K}$. The C*-algebra $p(B\otimes \K)p$ is $A$-absorbing by \cite[Cor.3.1]{paper:TomsWinter}. It follows that the $*$-homomorphisms $\varphi,\psi:A \to p(B\otimes \K)p$ are homotopic by Theorem~\ref{thm:cancellation}(b).

(ii) It suffices to prove approximate unitary equivalence for the restrictions of $\varphi$ and $\psi$
to $A\otimes M_n(\C)$ for any $n\geq 1$.  Let $(e_{ij})$ denote the canonical matrix unit
of $M_n(\C)$, $p_{ij}=\varphi(1\otimes e_{ij})$, $q_{ij}=\psi(1\otimes e_{ij})$, $p_n=\varphi(1\otimes 1_n)$
and $q_n=\psi(1\otimes 1_n)$.
By reasoning as in part (a), we find a partial isometry $v\in B\otimes \K$ such that $v^*v=p_{11}$
and $vv^*=q_{11}$. By \cite[Lemma 1.10]{paper:Mingo} there is a partial isometry $w\in M(B\otimes \K)$
such that $w^*w=1-p_n$ and $ww^*=1-q_n$. It follows that $V=w+\sum_{k=1}^n q_{k1}vp_{1k}$ is a unitary
in $M(B\otimes \K)$ such that $V\varphi(1\otimes x)V^*=\psi(1\otimes x)$ for all $x\in M_n(\C)$.
Thus
after conjugating $\varphi$ by a unitary  we may assume that
$\varphi(1\otimes x)=\psi(1\otimes x)$ for all $x\in M_n(\C)$.
Let us observe that if $a\in A$ and $u\in U(p_{11}(B\otimes \K)p_{11})$, then
$U=(1-p_n)+\sum_{k=1}^n p_{k1}up_{1k}\in U(M(B\otimes \K))$ satisfies
$U\varphi(a\otimes e_{ij})U^*-\psi(a\otimes e_{ij})=p_{i1}\big(u\varphi(a\otimes e_{11})u^*-\psi(a\otimes e_{11})\big)p_{1j}$.
This reduces our task
to proving approximate unitary equivalence for the unital maps $A\otimes e \to  p(B\otimes \K)p$
induced by  $\varphi$ and $\psi$, where $p=\varphi(1\otimes e)$.
Since $p(B\otimes \K)p$ is $A$-absorbing, this follows from Theorem~\ref{thm:cancellation}(b).
\end{proof}

Next we consider the case  when $B=C(X)\otimes A$ in Lemma~\ref{lem:homotopy-of-endos}.
We compare  two natural multiplicative $H$-space structures on $\Endo{A \otimes \K}$.

\begin{lemma}\label{lem:EckmannHilton1}
Let $X$ be a topological space, let $A$ be a  strongly self-absorbing $C^*$-algebra and let $\psi \colon A \otimes \K \to (A \otimes \K) \otimes (A \otimes \K)$ be a $*$-isomorphism. The two operations $\ast$ and $\circ$ on $G = [X,\Endo{A \otimes \K}]$  defined by
\begin{align*}
	[\alpha] \ast [\beta] = [\psi^{-1} \circ (\alpha \otimes \beta) \circ \psi] \quad \text{and} \quad [\alpha] \circ [\beta] = [\alpha \circ \beta]\ ,
\end{align*}
where $\alpha \otimes \beta \colon X \to \Endo{(A \otimes \K)^{\otimes 2}}$ denotes the pointwise tensor product, agree and are both associative and commutative. Moreover, $\ast$ does not depend on the choice of $\psi$ and is a group operation when restricted to $\Aut{A \otimes \K}$.
\end{lemma}

\begin{proof}\label{pf:EckmannHilton1}
First, let $\hat{\psi}$, $\hat{l}$ and $\hat{r}$ be as in the proof of Theorem \ref{thm:stabContractible} and use $\psi =\hat{\psi}(\frac{1}{2})$ in the definition of the operation $\ast$. Given $\alpha, \beta, \delta$ and $\gamma \in C(X,\Endo{A \otimes \K})$ we have
\begin{align*}
	([\alpha] \ast [\beta]) \circ ([\gamma] \ast [\delta]) & = [\psi^{-1} \circ (\alpha \otimes \beta) \circ \psi \circ \psi^{-1} \circ (\gamma \otimes \delta) \circ \psi] \\
	&= [\psi^{-1} \circ ((\alpha \circ \gamma) \otimes (\beta \circ \delta)) \circ \psi] = ([\alpha] \circ [\gamma]) \ast ([\beta] \circ [\delta])\ .
\end{align*}
Thus, the Eckmann-Hilton \cite{paper:Eckmann-Hilton} argument will imply that $\ast$ and $\circ$ agree and are both associative and commutative for this particular choice of $\psi$ if we can show that $\id{A \otimes \K}$ is a unit for the operation $\ast$. Just as in the proof of Theorem \ref{thm:stabContractible} we can see that $\hat{\psi}(t/2)^{-1} \circ (\alpha \otimes \id{A \otimes \K}) \circ \hat{\psi}(t/2)$, $t\in[0,1]$, is a homotopy from $\alpha$ to $\psi^{-1} \circ (\alpha \otimes \id{A \otimes \K}) \circ \psi$ with respect to the point-norm topology on $\Endo{A \otimes \K}$ proving that $\id{A \otimes \K}$ is a right unit. The analogous argument for $\hat{\psi}((t+1)/2)^{-1} \circ (\id{A \otimes \K} \otimes \alpha) \circ \hat{\psi}((t+1)/2)$ shows that $\id{A \otimes \K}$ is also a left unit.

If $\psi$ is chosen arbitrarily, we have $\psi = \hat{\psi}(\frac{1}{2}) \circ \kappa$ for some $\kappa \in \Aut{A \otimes \K}$. We denote the corresponding operations by $\ast_{\psi}$ and $\ast_{\hat{\psi}}$ and have
\[
	[\alpha] \ast_{\psi} [\beta] = [\kappa^{-1} \circ \hat{\psi}^{-1}(\tfrac{1}{2}) \circ (\alpha \otimes \beta) \circ \hat{\psi}(\tfrac{1}{2}) \circ \kappa] = [\kappa^{-1}] \circ ([\alpha] \ast_{\hat{\psi}} [\beta]) \circ [\kappa] = [\alpha] \ast_{\hat{\psi}} [\beta]
\]
by the homotopy commutativity of $\circ$. This proves the independence of $\ast$ from the choice of the isomorphism $\psi$.
\end{proof}
We denote by $\approx_{u}$ the relation of approximate unitary equivalence
for $*$-homomorphisms.
\begin{lemma}\label{lemma:p-tensor-q} Let $A$ be a strongly self-absorbing C*-algebra.
If $p\in A\otimes \K$ is a nonzero projection, the following conditions are equivalent:
\begin{itemize}
\item[(i)] $p(A\otimes \K)p\cong A$
\item[(ii)] There is $\alpha\in \Aut{A\otimes \K}$ such that $\alpha(1\otimes e)=p$.
\item[(iii)] $[p]\in K_0(A)^\times_{+}$
\end{itemize}
We denote by $\Proj{A \otimes \K}^\times$  the set of all projections  satisfying
these equivalent conditions.
\end{lemma}
\begin{proof}  (i) $\Leftrightarrow$ (ii) This follows from Corollary~\ref{Kodaka}.

(ii) $\Rightarrow$ (iii)
 As an immediate consequence of Lemma~\ref{lem:EckmannHilton1} one verifies
that the map  $\Theta:\pi_0(\mathrm{End}(A\otimes \K))\to K_0(A)$,
$\Theta[\alpha]=[\alpha(1\otimes e)]$ is multiplicative, i.e. $\Theta [\alpha\circ \beta]=\Theta [\alpha]\Theta[\beta]$.
Let $q:=\alpha^{-1}(1\otimes e)$. Then
$[p][q]=\Theta[\alpha\circ \alpha^{-1}]=\Theta[\mathrm{id}]=[1]$.

(iii) $\Rightarrow$ (i) By assumption there is a full projection $q\in A\otimes K$ such that
$[p][q]=[1]$ in $K_0(A)$. By Lemma~\ref{lem:homotopy-of-endos} there are
$\varphi,\psi\in \mathrm{End}(A\otimes \K)$ such that $\varphi(1\otimes e)=p$
and
$\psi(1\otimes e)=q$. Since $[p][q]=[1]$ in $K_0(A)$, it follows that $[\varphi\circ \psi]=[\psi\circ \varphi]=[\mathrm{id_{A\otimes \K}}]\in [A\otimes \K,A\otimes \K]$. Therefore $\varphi\circ \psi \approx_{u} \mathrm{id}_{A\otimes \K} \approx_{u} \psi\circ \varphi$ by Proposition~\ref{prop:homotopy-of-endoss}.
By \cite[Cor.2.3.4]{Ror:encyclopedia} it follows that there is an automorphism $\varphi_0\in \Aut{A\otimes \K}$
such that $\varphi_0 \approx_{u} \varphi$. Set $p_0=\varphi_0(1_A\otimes e)$.
The map $\varphi_0$ induces a $*$-isomorphism
$A\cong (1_A\otimes e) (A\otimes \K)(1_A\otimes e)\to p_0(A\otimes \K)p_0$. We conclude
that $A\cong p(A\otimes \K)p$ since $p_0$ is unitarily equivalent to $p$.
\end{proof}
If $A$ is a separable unital C*-algebra,   Brown, Green and Rieffel \cite{paper:BGR} showed that
the Picard group   $\mathrm{Pic}(A)$ is isomorphic to the outer automorphism group of $A\otimes \K$, i.e.
 $\mathrm{Pic}(A)\cong \mathrm{Out}(A\otimes \K)=\Aut{A\otimes \K}/\mathrm{Inn}(A\otimes \K)$. One can view $\mathrm{Out}(A)$ as a subgroup of $\mathrm{Pic}(A)$. Kodaka \cite{paper:Kodaka} has shown that the coset space
 $\mathrm{Pic}(A)/\mathrm{Out}(A)$ is in bijection with the Murray-von Neumann equivalence classes of full projections $p\in A\otimes \K$ such that $p(A\otimes \K)p\cong A$.
 From Lemmas~\ref{lem:EckmannHilton1} and \ref{lemma:p-tensor-q} we see that if $A$ is strongly
 self-absorbing, then $\mathrm{Out}(A)$ is a normal subgroup of $\mathrm{Pic}(A)$ and we have:
\begin{corollary}\label{cor:Picard}
If $A$ is strongly self-absorbing, then there is an exact sequence of groups
 \[1\to \mathrm{Out}(A) \to \mathrm{Pic}(A)\to K_0(A)^{\times}_{+}\to 1.\]
\end{corollary}
If moreover $A$ is stably finite, then its normalized trace induces a  homomorphism of multiplicative groups from $K_0(A)^{\times}_{+}$ onto
the fundamental group $\mathcal{F}(A)$ of $A$ defined in \cite{paper:Nawata-Watatani}.
 \begin{lemma}\label{lem:full-AutAKfibration} Let $A$ be a strongly self-absorbing C*-algebra.
The sequences
\(
	\AutSt{A \otimes \K}{1 \otimes e} \to \Aut{A \otimes \K} \to \Proj{A \otimes \K}^\times
\)
and
\(
	\EndSt{A \otimes \K}{1 \otimes e} \to \Endo{A \otimes \K} \to \Proj{A \otimes \K}^\times
\)
where the first map is the inclusion and  the second sends $\alpha$ to $\alpha(1 \otimes e)$ is a locally trivial fiber bundle over a paracompact base space and therefore it is  a Hurewicz fibration.
\end{lemma}
\begin{proof}  Lemma~\ref{lemma:p-tensor-q} shows that
 the map to the base space is surjective.
With this remark, the proof is entirely similar the proof of Lemma~\ref{lem:AutAKfibration}.
\end{proof}
\begin{corollary} \label{cor:AutAK_BUA}
Let $A$ be a  strongly self-absorbing  $C^*$-algebra. Then $\Aut{A \otimes \K} \simeq \Endo{A \otimes \K}^{\times}$ has the homotopy type of a CW complex, which is {homotopy equivalent} to $K_0(A)_{+}^{\times} \times BU(A)$.
\end{corollary}

\begin{proof}
The equivalence $\Aut{A \otimes \K} \simeq \Endo{A \otimes \K}^{\times}$ follows from lemma \ref{lem:full-AutAKfibration} and Theorem \ref{thm:stabContractible}. Moreover,	
$\Aut{A \otimes \K}$ is the coproduct of its path components, all of which are homeomorphic
to $\uAut{A \otimes \K}$.
By  Theorem~\ref{thm:stabContractible} and Lemma~\ref{lem:AutAKfibration},
$\uAut{A \otimes \K}$ is homotopy equivalent to $\uProj{A \otimes \K}$.
By Lemma~\ref{lemma:p-tensor-q}, $\pi_0(\Proj{A \otimes \K}^\times)\cong K_0(A)^\times_+$.
Thus, using Corollary~\ref{cor:CWcomplex} we have
\[
\Aut{A \otimes \K}\simeq \pi_0(\Proj{A \otimes \K}^\times)\times \uProj{A \otimes \K}
\simeq K_0(A)^{\times}_+\times BU(A).
\qedhere\]
\end{proof}
In the case $A=\mathbb{C}$ this reproves the well-known
fact that $\Aut{\K}\simeq BU(1)\simeq K(\mathbb{Z},2)$
and hence the only non vanishing homotopy group of $\Aut{\K}$ is $\pi_2(\Aut{\K})\cong \pi_2(BU(1))\cong \pi_1(U(1))\cong \mathbb{Z}$. At the same time, for $A\neq \C$, we obtain the following.
\begin{theorem}\label{thm:homotopy-groups}  Let $A\neq \C$ be a strongly self-absorbing C*-algebra. Then there are isomorphisms of groups
\[\pi_i(\Aut{A\otimes \K})=
\begin{cases}
		K_0(A)_{+}^{\times} & \text{if } i=0 \\
			K_i(A) & \text{if } i\geq 1\ .
	\end{cases}
\]
\end{theorem}
\begin{proof} We have seen in the proof of Corollary~\ref{cor:AutAK_BUA} that $\pi_0(\Aut{A\otimes \K})\cong K_0(A)_{+}^{\times}$. If $i\geq 1$,
then by Corollary~\ref{cor:AutAK_BUA},
 $\pi_i(\Aut{A\otimes \K})\cong \pi_i (BU(A))\cong \pi_{i-1}(U(A))$.
On the other hand, since $A$ is $\ZZ$-stable, we have that $\pi_{i-1}(U(A))\cong K_i(A)$ by \cite[Thm.3]{paper:Jiang-nonstable}.
\end{proof}
\begin{corollary}\label{cor:101} Let $A$ be a strongly self-absorbing $C^*$-algebra. There is an exact sequence of
topological groups $1\to \mathrm{Aut}_{0}(A\otimes \K)\to \mathrm{Aut}(A\otimes \K)\to K_0(A)^{\times}_{+}\to 1$.
\end{corollary}
\begin{remark}\label{rem:symmetry} The exact sequence $1\to \mathrm{Aut}_{0}(\OO_\infty\otimes \K)\to \mathrm{Aut}(\OO_\infty\otimes \K)\to\mathbb{Z}/2\to 1$ is split, since by \cite{paper:Benson-Kumjian-Phillips} there is an order-two automorphism $\alpha$
of $\OO_\infty\otimes \K$ such that $\alpha_*=-1$ on $K_0(\OO_\infty)$.
\end{remark}
\begin{corollary}\label{cor:102} Let $A\neq \C$ be a strongly self-absorbing $C^*$-algebra.
The natural map \newline $\mathrm{Aut}_{0}(A\otimes \K)\to \mathrm{Aut}_{0}(\OO_\infty\otimes A\otimes \K)$  is
a homotopy equivalence.
\end{corollary}
\begin{proof} $\mathrm{Aut}_{0}(A\otimes \K)\to \mathrm{Aut}_{0}(\OO_\infty\otimes A  \otimes \K)$ is given by
$\alpha\mapsto \mathrm{id}_{\OO_\infty}\otimes \alpha$.
Both spaces have the homotopy type of a CW-complex by Corollary~\ref{cor:CWcomplex} and they are {weakly homotopy equivalent} by Thm.~\ref{thm:homotopy-groups}.
\end{proof}

\begin{theorem}\label{thm:MapsEndAK}
Let $A\neq \C$ be a strongly self-absorbing $C^*$-algebra and let $X$ be a  compact metrizable space.
The map $\Theta \colon [X, \Endo{A \otimes \K}] \to K_0(C(X) \otimes A)_+$ given by $\Theta([\alpha]) = [\alpha(1 \otimes e)]$ is an isomorphism of commutative semirings.
  $\Theta$ restricts to a group isomorphism $[X, \Aut{A \otimes \K}] \to K_0(C(X) \otimes A)_+^{\times}$.
  If $X$ is connected, then
$K_0(C(X) \otimes A)_+^{\times}\cong K_0(A)^{\times}_{+}\oplus K_0(C_0(X\setminus x_0)\otimes A).$
If  $A$ is purely infinite, then $K_0(C(X) \otimes A)_+=K_0(C(X) \otimes A)$
and $\Theta$ is an isomorphism of  rings.
\end{theorem}
\begin{proof} \label{pf:MapsEndAK}
Let $\psi \colon A \otimes \K \to (A \otimes \K)^{\otimes 2}$ and $l$ be as in Lemma \ref{lem:EckmannHilton1}.
 The additivity of $\Theta$ is easily verified.
  Let $\alpha, \beta \in C(X,\Endo{A \otimes \K})$, then by Lemma \ref{lem:EckmannHilton1}:
\begin{align*}
	[(\alpha \circ \beta)(1 \otimes e)] &= [(\alpha \ast \beta)(1 \otimes e)] = [\psi^{-1} \circ (\alpha \otimes \beta) \circ \psi(1 \otimes e)] \\
	& = [\psi^{-1}(\alpha(1 \otimes e) \otimes \beta(1 \otimes e))] = [\alpha(1 \otimes e)] \cdot [\beta(1 \otimes e)]\ ,
\end{align*}
which shows that $\Theta \colon [X,\Endo{A \otimes \K}] \to K_0(C(X) \otimes A)_+$ is a homomorphism of semirings.  Let $p\in C(X)\otimes A\otimes \K$ be a full projection. Then,  $p(C(X)\otimes A\otimes \K)p$ is $A$-absorbing by \cite[Cor.3.1]{paper:TomsWinter}. It follows that
  $\Theta$ is surjective by Lemma~\ref{lem:homotopy-of-endos}. For injectivity we apply Proposition~\ref{prop:homotopy-of-endoss}(i).

Next we show that the image of the restriction of $\Theta$ to $[X,\Aut{A\otimes \K}]$
coincides with $K_0(C(X)\otimes \K)^\times_{+}$. Let $p\in \Proj{C(X)\otimes A \otimes \K}^\times$.
By assumption, there is $q\in \Proj{C(X)\otimes A \otimes \K}^\times$ such that $[p][q]=1$ in the ring $K_0(C(X)\otimes A)$.
By  Lemma~\ref{lem:homotopy-of-endos} there are $\varphi,\psi \in \mathrm{Hom}(A\otimes \K, C(X)\otimes A \otimes \K)$ such that $\varphi(1\otimes e)=p$
and $\psi(1\otimes e)=q$. Let $\tilde{\varphi},\tilde{\psi}\in \mathrm{End}_{C(X)}(C(X)\otimes A\otimes \K)$
be the unique $C(X)$-linear extensions of $\varphi$ and $\psi$. Note that if $\iota:A\otimes \K\to C(X)\otimes A \otimes \K$
is the inclusion $\iota(a)=1_{C(X)}\otimes a$, then $\tilde{\iota}=\id{C(X)\otimes A \otimes \K}$.
Since $[p][q]=[1]$ in $K_0(C(X)\otimes A)$, it follows that $[\tilde{\varphi}\circ \psi]=[\tilde{\psi}\circ \varphi]=[\iota]\in [ A\otimes \K,C(X)\otimes A\otimes \K]$. By Proposition~\ref{prop:homotopy-of-endoss}(ii) it follows that $\tilde{\varphi}\circ \psi \approx_{u} \iota \approx_{u} \tilde{\psi}\circ \varphi$. This clearly implies that $\tilde{\varphi}\circ \tilde{\psi} \approx_{u} \id{C(X)\otimes A \otimes \K} \approx_{u} \tilde{\psi}\circ \tilde{\varphi}$.
By \cite[Cor.2.3.4]{Ror:encyclopedia} it follows that there is an automorphism $\alpha\in \mathrm{Aut}_{C(X)}(C(X)\otimes A\otimes \K)$
such that $\alpha \approx_{u} \tilde{\varphi}$. In particular we have that $[\alpha(1\otimes e)]=[\tilde{\varphi}(1\otimes e)]=[p]$.

It remains to verify the isomorphism $K_0(C(X)\otimes A)^{\times}_{+}\cong K_0(A)^{\times}_{+}\oplus K_0(C_0(X\setminus x_0)\otimes A).$
Evaluation at $x_0$ induces a split exact sequence
\(
0\to K_0(C_0(X\setminus x_0)\otimes A) \to K_0(C(X)\otimes A) \to K_0(A)\to 0.
\)
Arguing as in the proof of \cite[Prop.5.6]{paper:DadarlatPiAutAK}, one verifies that $K_0(C_0(X\setminus x_0)\otimes A)$ is a nil-ideal of the ring $K_0(C(X)\otimes A)$. Thus an element $\sigma\in K_0(C(X)\otimes A)$ is invertible if and only if its restriction $\sigma_{x_0}\in K_0(A)$ is invertible. Consequently
\(
K_0(C(X)\otimes A)^{\times}\cong K_0(A)^{\times}\oplus K_0(C_0(X\setminus x_0)\otimes A).
\)
It remains to verify that an element $\sigma\in K_0(C(X)\otimes A)$ is positive if $\sigma_{x_0}\in K_0(A)_{+}\setminus \{0\}$. It suffices to consider the case when $A$ is stably finite. Let $\tau$ denote the unique
trace  state of $A$.  Its extension to a trace state on  $A\otimes M_n(\C)$ is denoted again by $\tau$.
Then any continuous trace $\eta$ on $C(X)\otimes A \otimes M_n(\mathbb{C})$ is of the form
$\eta(f)=\int_X \tau(f)d\mu$ for some finite Borel measure $\mu$ on $X$.
Write $\sigma=[p]-[q]$ where  $p,q\in \Proj{C(X)\otimes A \otimes M_n(\mathbb{C})}$ are full projections.
Let    $r$ be a nonzero projection in $A\otimes M_n(\mathbb{C})$ such that $[p(x_0)]-[q(x_0)]=[r]$.
Since $X$ is connected it follows that $[p(x)]-[q(x)]=[r]\in K_0(A)$ for all $x\in X$.
From this we see that  any point $x\in X$ has a closed neighborhood $V$ such that
$[p_V]-[q_V]=[r]\in K_0(C(V)\otimes A)$. Since $\tau(r)>0$ it follows immediately
that $\eta(p)>\eta(q)$ for all nonzero finite traces $\eta$ on  $A\otimes M_n(\C)$.
We apply Corollaries 4.9 and 4.10 of \cite{paper:Rordam-ZStable} to conclude that $[p]-[q]\in K_0(C(X)\otimes A)_{+}$.
\end{proof}

\begin{corollary} \label{cor:Jiang-Su/Cuntz-infinity}
Let $X$ be a compact  connected metrizable space. Then there are isomorphism of multiplicative groups
\[
[X,\Aut{\ZZ \otimes \K}]\cong K^0(X)^\times_{+}=1+\widetilde{K}^0(X),
\]
\[
[X,\Aut{\OO_\infty \otimes \K}]\cong K^0(X)^\times=\pm 1+\widetilde{K}^0(X).
\]
\end{corollary}

\subsection{The topological group $\Aut{A \otimes \K}$ is well-pointed}
Since we would like to apply the nerve construction to obtain classifying spaces of the topological monoids $\Aut{A \otimes \K}$ and $\Endo{A \otimes \K}^{\times}$, we will need to show that $\Aut{A \otimes \K}$ is well-pointed. This notion is defined as follows:

\begin{definition}\label{def:wellpointed}
Let $X$ be a topological space, $A \subset X$ a closed subspace. The pair $(X, A)$ is called a \emph{neighborhood deformation retract} (or \emph{NDR-pair} for short) if there is a map $u \colon X \to I=[0,1]$ such that $u^{-1}(0) = A$ and a homotopy $H \colon X \times I \to X$ such that $H(x,0) = x$ for all $x\in X$, $H(a,t) = a$ for $a \in A$ and $t \in I$ and $H(x,1) \in A$ if $u(x) < 1$. A pointed topological space $X$ with basepoint $x_0 \in X$ is said to have a \emph{non-degenerate basepoint} or to be \emph{well-pointed} if the pair $(X,x_0)$ is an NDR-pair.
\end{definition}

Recall that a neighborhood  $V$ of $x_0$ deformation retracts to $x_0$
if there is a continuous map $h:V \times I \to V$ such that $h(x,0)=x$, $h(x_0,t)=x_0$ and $h(x,1)=x_0$
for all $x\in V$ and $t\in I$. The following lemma is contained in \cite[Thm.2]{paper:A.Strom}.
\begin{lemma}\label{lem:NDRcontractible}
Let $(X,x_0)$ be a pointed topological space together with a continuous map $v \colon X \to I$  such that $x_0=v^{-1}(0)$ and $V = \{x \in X \colon v(x)<1\}$ deformation retracts to $x_0$. Then $(X,x_0)$ is an NDR-pair.
\end{lemma}

\begin{proposition}\label{prop:wellpointed}
Let $A$ be a strongly self-absorbing $C^*$-algebra. Then the topological monoids $\Aut{A \otimes \K}$ and $\Endo{A \otimes \K}^{\times}$ are well-pointed.
\end{proposition}

\begin{proof}
We will prove this for $\Aut{A \otimes \K}$, but the proof for $\Endo{A \otimes \K}^{\times}$ is entirely similar.
Let $e \in \K$ be a rank-$1$ projection and set $p_0=1\otimes e$. Let $U = \{ p \in \Proj{A \otimes \K}\ | \ \lVert p - p_0 \rVert < 1/2\}.$ If $\pi \colon \uAut{A \otimes \K} \to \uProj{A \otimes \K}$ denotes the map $\beta\mapsto \beta(p_0)$, we will show that $\pi^{-1}(U)$ deformation retracts to $\id{A \otimes \K} \in \uAut{A \otimes \K}$.
As we have seen in the proof of Lemma \ref{lem:AutAKfibration}, the principal bundle $\uAut{A \otimes \K} \to \uProj{A \otimes \K}$ trivializes over $U$, i.e.\ there exists a homeomorphism
\(
	\pi^{-1}(U) \to U \times \AutSt{A \otimes \K}{p_0}
\)
sending $\id{A \otimes \K}$ to $(p_0, \id{A \otimes \K})$. Thus, it suffices to show that the right hand side retracts. Let $\chi$ be the characteristic function of $(1/2,1]$. Then $h(p,t)=\chi((1-t)p+tp_0)$ is a  deformation retraction of $U$ into $p_0$. This is well-defined since $1/2$ is not in the spectrum of $a= (1-t)p+tp_0$ as seen from the estimate $\|(1-2a)-(1-2p_0)\|<1$. We have shown in Theorem \ref{thm:stabContractible} that $\AutSt{A \otimes \K}{p_0}$ deformation retracts to $\id{A \otimes \K}$.  Combining these homotopies we end up with a deformation retraction of $\pi^{-1}(U)$ into $\id{A \otimes \K}$.
Let $d$ be a metric for $\Aut{A \otimes \K}$.
Then
	$v\colon \Aut{A \otimes \K} \to [0,1]$,
	$v(\alpha) = \max\{\min\{d(\alpha,\id{A\otimes \K}),1/2\}, \min\{1, 2\lVert \alpha(p_0) - p_0 \rVert \}\}$ and $V:=\pi^{-1}(U)$
satisfy the conditions of Lemma \ref{lem:NDRcontractible} relative to the basepoint $\id{A \otimes \K}$.
\end{proof}

\section{The infinite loop space structure of $B\Aut{A \otimes \K}$}

\subsection{Permutative categories and infinite loop spaces}
We will show that $B\Aut{A \otimes \K}$ is an infinite loop space in the sense of the following definition \cite{book:Adams}.

\begin{definition}\label{def:infinite_loop_space}
A topological space $E = E_0$ is called an \emph{infinite loop space}, if there exists a sequence of spaces $E_i$, $i \in \N$, such that $E_i \simeq \Omega E_{i+1}$ for all $i \in \N_0$ ($\simeq$ denotes homotopy equivalence).
\end{definition}

The importance of these spaces lies in the fact, that they represent generalized cohomology theories, i.e.\ for a CW-complex $X$, the homotopy classes of maps $E^i(X) := [X,E_i]$ are abelian groups and the functor $X \mapsto E^{\bullet}(X)$ is a cohomology theory. {There may be many inequivalent delooping sequences starting with the same $E_0$ leading to different theories.} The sequence of spaces $E_i$ forms a \emph{connective $\Omega$-spectrum}. There is a well-developed theory to detect whether a space belongs to this class \cite{paper:May, paper:SegalCatAndCoh}. One of the main sources for infinite loop spaces are classifying spaces of topological strict symmetric monoidal categories, called \emph{permutative categories} in \cite{paper:MayPermutative}. 

A \emph{topological category} has a space of objects, a space of morphisms and continuous source, target and identity maps. Such a category $\mathcal{C}$ carries a \emph{strict monoidal structure} if it comes equipped with a functor $\otimes \colon \mathcal{C} \times \mathcal{C} \to \mathcal{C}$ that satisfies the analogues of associativity and unitality known for monoids. The strictness refers to the fact that these hold on the nose, not only up to natural transformations. $\mathcal{C}$ is called \emph{symmetric} if it comes equipped with a natural transformation $c \colon \otimes \circ \tau \to \otimes$, where $\tau \colon \mathcal{C} \times \mathcal{C} \to \mathcal{C} \times \mathcal{C}$ is switching the factors. This should behave like a permutation on $n$-fold tensor products. We will assume all permutative categories to be \emph{well-pointed} in the sense that the map ${\rm obj}(\mathcal{C}) \to {\rm mor}(\mathcal{C}), \ x \to \id{x}$ is a cofibration. For a precise definition, we refer the reader to \cite[Def.1]{paper:MayPermutative}. {Note further that all categories we consider in this paper will be \emph{small}.}

Any topological category $\mathcal{C}$ can be turned into a simplicial space $N_{\bullet}\mathcal{C}$ via the nerve construction. Let $N_0\,\mathcal{C} = {\rm obj}(\mathcal{C})$, $N_1\mathcal{C} = {\rm mor}(\mathcal{C})$ and
\[
	N_k\,\mathcal{C} = \{ (f_1, \dots, f_k) \in {\rm mor}(\mathcal{C}) \times \dots \times {\rm mor}(\mathcal{C})\ |\ s(f_i) = t(f_{i+1})\}\ .
\]
The face maps $d_i^k \colon N_k\,\mathcal{C} \to N_{k-1}\,\mathcal{C}$ and degeneracies $s_i^k\colon N_k\,\mathcal{C} \to N_{k+1}\,\mathcal{C}$ are induced by composition of successive maps and insertion of identities respectively. The geometric realization of a simplicial space $X_{\bullet}$ is defined by
\[
	\lvert X_{\bullet} \rvert = \left(\coprod_{k = 0} \,X_k \times \Delta_k\right) / \sim
\]
where $\Delta_k \subset \R^{k+1}$ denotes the standard $k$-simplex and the equivalence relation is generated by $(d_i^kx, u) \sim (x, \partial_i^{k-1}u)$ and $(s_i^ly, v) \sim (y, \sigma_i^{l+1}v)$ for $x \in X_k$, $u \in \Delta_{k-1}$, $y \in X_l$, $v \in X_{l+1}$, where $\delta_i$ and $\sigma_i$ are the coface and codegeneracy maps on the standard simplex. For details about this construction we refer the reader to \cite[sec.11]{paper:May}.

The space $\lvert N_{\bullet} \mathcal{C} \rvert$ associated to a category $\mathcal{C}$ is called the \emph{classifying space of $\mathcal{C}$}. If $\mathcal{C}$ is the category associated to a monoid $M$, then we denote $\lvert N_{\bullet} \mathcal{C} \rvert$ by $BM$. Having a monoidal structure on $\mathcal{C}$ yields the following.

\begin{lemma} \label{lem:top_monoid}
Let $\mathcal{C}$ be a strict monoidal topological category. Then $\lvert N_{\bullet} \mathcal{C} \rvert$ is a topological monoid.
\end{lemma}

\begin{proof}\label{pf:top_monoid}
The nerve construction $N_{\bullet}$ preserves products in the sense that the projection functors $\pi_i \colon \mathcal{C} \times \mathcal{C} \to \mathcal{C}$ induce a levelwise homeomorphism $N_{\bullet}(\mathcal{C} \times \mathcal{C}) \to N_{\bullet}\mathcal{C} \times N_{\bullet} \mathcal{C}$. Therefore $N_{\bullet}\mathcal{C}$ is a simplicial topological monoid and the lemma follows from \cite[Cor.11.7]{paper:May}.
\end{proof}

A permutative category $\mathcal{C}$ provides an input for infinite loop space machines \cite[Def.2]{paper:MayPermutative}. Due to the above lemma, there is a classifying space $B\lvert N_{\bullet}\mathcal{C} \rvert$. The following has been proven by Segal \cite{paper:SegalCatAndCoh} and May \cite[Thm.4.10]{paper:MayGroupCompletion}

\begin{theorem} \label{thm:InfLoopSpace}
Let $\mathcal{C}$ be a permutative category. Then $\Omega B\lvert N_{\bullet} \mathcal{C}\rvert$ is an infinite loop space. Moreover, if $\pi_0(\lvert N_{\bullet} \mathcal{C} \rvert)$ is a group, then the map $\lvert N_{\bullet}\mathcal{C} \rvert \to \Omega B\lvert N_{\bullet} \mathcal{C} \rvert$ induced by the inclusion of the $1$-skeleton $S^1 \times \lvert N_{\bullet}\mathcal{C}\rvert \to B\lvert N_{\bullet} \mathcal{C} \rvert$ is a homotopy equivalence of $H$-spaces.
\end{theorem}

\subsection{The tensor product of $A \otimes \K$-bundles}
Let $A$ be a  strongly self-absorbing $C^*$-algebra, $X$ be a topological space and let $P_1$ and $P_2$ be principal $\Aut{A \otimes \K}$-bundles over $X$. Fix an isomorphism
$\psi \colon A \otimes \K \to (A \otimes \K) \otimes (A \otimes \K)$.
This choice induces a tensor product operation on principal $\Aut{A \otimes \K}$-bundles in the following way. Note that $P_1 \times_X P_2 \to X$ is a principal $\Aut{A \otimes \K} \times \Aut{A \otimes \K}$-bundle and $\psi$ induces a group homomorphism
\[
	{\rm Ad}_{\psi^{-1}} \colon \Aut{A \otimes \K} \times \Aut{A \otimes \K} \to \Aut{A \otimes \K} \quad ; \quad (\alpha,\beta) \mapsto \psi^{-1} \circ (\alpha \otimes \beta) \circ \psi\ .
\]
Now let
\[
	P_1 \otimes_{\psi} P_2 := (P_1 \times_X P_2) \times_{{\rm Ad}_{\psi^{-1}}} \Aut{A \otimes \K} = ((P_1 \times_X P_2) \times \Aut{A \otimes \K})/\sim
\]
where the equivalence relation is $(p_1\,\alpha, p_2\,\beta, \gamma) \sim (p_1, p_2, {\rm Ad}_{\psi^{-1}}(\alpha,\beta)\,\gamma)$ for all $(p_1, p_2) \in P_1 \times_X P_2$ and $\alpha, \beta, \gamma \in \Aut{A \otimes \K}$. This is a delooped version of the operation $\ast$ from Lemma \ref{lem:EckmannHilton1}.

Due to the choice of $\psi$, which was arbitrary, $\otimes_{\psi}$ can not be associative. We will show, however, that  -- just like $\ast$ -- it is \textit{homotopy associative} and also \textit{homotopy unital}.

To obtain a model for the classifying space $B\Aut{A \otimes \K}$, let $\mathcal{B}$ be the topological category, which has as its object space just a single point and the group $\Aut{A \otimes \K}$ as its morphism space. Since we have shown that $\Aut{A \otimes \K}$ is well-pointed (see Proposition \ref{prop:wellpointed}), \cite[Prop.7.5 and Thm.8.2]{book:MayClassifying} implies that the geometric realization $\lvert N_{\bullet}\mathcal{B} \rvert$ has in fact the homotopy type of a classifying space for principal $\Aut{A \otimes \K}$-bundles, i.e.\
\[
	B\Aut{A \otimes \K} = \lvert N\mathcal{B}_{\bullet} \rvert\ .
\]
Choosing an isomorphism $\psi \colon A \otimes \K \to (A \otimes \K) \otimes (A \otimes \K)$, we can define a functor $\otimes_{\psi} \colon \mathcal{B} \times \mathcal{B} \to \mathcal{B}$ just as above, which acts on morphisms $\alpha, \beta \in \Aut{A \otimes \K}$ by
\[
 	(\alpha, \beta) \mapsto \alpha \otimes_{\psi} \beta := {\rm Ad}_{\psi^{-1}}(\alpha, \beta)\ .
\]
This is in fact functorial since composition is well-behaved with respect to the tensor product in the following way 	
\begin{align*}\label{eqn:EHcondition}
	& (\alpha \circ \alpha') \otimes_{\psi} (\beta \circ \beta') = \psi^{-1} \circ (\alpha \circ \alpha') \otimes (\beta \circ \beta') \circ \psi \\  = &\  \psi^{-1} \circ (\alpha \otimes \beta) \circ \psi \circ \psi^{-1} \circ (\alpha' \otimes \beta') \circ \psi = (\alpha \otimes_{\psi} \beta) \circ (\alpha' \otimes_{\psi} \beta')\ .
\end{align*}
The functor induces a multiplication map on the geometric realization
\[
	\mu_{\psi} \colon B\Aut{A \otimes \K} \times B\Aut{A \otimes \K} \to B\Aut{A \otimes \K}\ .
\]
Observe that a path connecting two isomorphisms $\psi, \psi' \in {\rm Iso}(\Aut{A \otimes \K}, \Aut{A \otimes \K}^{\otimes 2})$ induces a homotopy of functors
$\mathcal{B} \times \mathcal{B} \times I \to \mathcal{B}$,
where $I$ here is the category, which has $[0,1]$ as its object space and only identities as morphisms. After geometric realization this in turn yields a homotopy between $\mu_{\psi}$ and $\mu_{\psi'}$ (observe that $\lvert I \rvert \cong [0,1]$).

\begin{lemma}\label{lem:HSpace}
Let $A$ be a  strongly self-absorbing $C^*$-algebra and let $\mathcal{B}$ be the category defined above. Let $\psi \colon A \otimes \K \to (A \otimes \K) \otimes (A \otimes \K)$ be an isomorphism, then $\mu_{\psi}$ defines an $H$-space structure on $B\Aut{A \otimes \K}$, which has the basepoint of $B\Aut{A \otimes \K}$ as a homotopy unit and agrees with the $H$-space structure induced by the tensor product $\otimes_{\psi}$ of $A \otimes \K$-bundles. Different choices of $\psi$ yield homotopy equivalent $H$-space structures.
\end{lemma}

\begin{proof}
The proof of this statement is very similar to the one of Lemma \ref{lem:EckmannHilton1}, but we have to take care that the homotopies we use run through \emph{functors} on $\mathcal{B}$. Let $\hat{\psi}$, $\hat{l}$ and $\hat{r}$ be just as in Theorem~\ref{thm:stabContractible} and consider $\psi = \hat{\psi}(\frac{1}{2})$ first. By Theorem \ref{thm:stabContractible}, there is a path between $(\psi \otimes \id{A \otimes \K}) \circ \psi$ and $(\id{A \otimes \K} \otimes \psi) \circ \psi$, since both these morphisms map $1\otimes e$ to $(1\otimes e)^{\otimes 3}$. This proves the homotopy associativity in this case.

To prove that the basepoint provides a homotopy unit we have to show that the two functors $\alpha \mapsto \alpha \otimes_{\psi} \id{A \otimes \K}$ and $\alpha \mapsto \id{A \otimes \K} \otimes_{\psi} \alpha$ are both homotopic to the identity functor. The argument for this is the same as in the proof of Lemma \ref{lem:EckmannHilton1}.

Now let $\psi \colon A \otimes \K \to (A \otimes \K) \otimes (A \otimes \K)$ be an arbitrary isomorphism. As in Lemma~\ref{lem:EckmannHilton1} we have $\psi = \hat{\psi} \circ \kappa$ for some automorphism $\kappa \in \Aut{A \otimes \K}$. If we denote homotopic functors by $\sim$, we have
\begin{align*}
	\alpha \otimes_{\psi} \beta & = \kappa^{-1} \circ (\alpha \otimes_{\hat{\psi}} \beta) \circ \kappa \sim (\kappa \otimes_{\hat{\psi}} \id{A \otimes \K})^{-1} \circ (\id{A \otimes \K} \otimes_{\hat{\psi}} (\alpha \otimes_{\hat{\psi}} \beta)) \circ (\kappa \otimes_{\hat{\psi}} \id{A \otimes \K}) \\
	& = \id{A \otimes \K} \otimes_{\hat{\psi}} (\alpha \otimes_{\hat{\psi}} \beta) \sim (\alpha \otimes_{\hat{\psi}} \beta)\ .
\end{align*}
Note that every stage of this homotopy provides functors $\mathcal{B} \times \mathcal{B} \to \mathcal{B}$. Geometrically realizing this homotopy we see that different choices of $\psi$ yield the same $H$-space structure up to homotopy.

Let $EG \to BG$ be the universal $G$-bundle where $G=\Aut{A \otimes \K}$ \cite[section 7]{book:MayClassifying}. Using its simplicial description, we see that $\mu_{\psi}^*EG \cong \pi_1^*EG \otimes_{\psi} \pi_2^*EG$, where $\pi_i \colon BG \times BG \to BG$ are the canonical projections. Now, given two classifying maps $f_k \colon X \to BG$ and the diagonal map $\Delta \colon BG \to BG \times BG$, we have
\begin{align*}
	(\mu_{\psi} \circ (f_1,f_2) \circ \Delta)^*EG & = \Delta^* \circ (f_1^*, f_2^*) \circ \mu_{\psi}^*EG = f_1^*EG \otimes_{\psi} f_2^*EG
\end{align*}
proving that the multiplication induced by the $H$-space structure on $[X, B\Aut{A \otimes \K}]$ agrees with the tensor product $\otimes_{\psi}$.
\end{proof}

\begin{definition} \label{def:BunXAutAK}
For a strongly self-absorbing $C^*$-algebra $A$ we define $(\BunX{A \otimes \K}, \otimes)$ to be the monoid of isomorphism classes of principal $\Aut{A \otimes \K}$-bundles with respect to the tensor product induced by $\otimes_{\psi}$. By the above lemma, this is independent of the choice of $\psi$.
\end{definition}
	
To apply the infinite loop space machine, we need a permutative category encoding the operation $\otimes_{\psi}$. Let $\mathcal{B}_{\otimes}$ be the category, which has $\N_0 = \{0,1,2, \dots\}$ as its object space (where $n \in \N_0$ should be thought of as~$(A \otimes \K)^{\otimes n}$ with $(A \otimes \K)^{\otimes 0} = \C$). The morphisms from $n$ to $m$ are given by
\(
	\hom_{\mathcal{B}_{\otimes}}(n,m) = \left\{ \alpha \in \mathrm{Hom}((A \otimes \K)^{\otimes n}, (A \otimes \K)^{\otimes m})\ |\ KK(\alpha)\ {\rm invertible}\right\}
\) for $n\geq 1$
  and  $\hom_{\mathcal{B}_{\otimes}}(0,m) = \left\{ \alpha \in \mathrm{Hom}(\C, (A \otimes \K)^{\otimes m})\ |\, \,[\alpha(1)]\in K_0( (A \otimes \K)^{\otimes m})^\times\right\}$ for $n=0$.
We equip these spaces with the point-norm topology. The ordinary tensor product of $*$-homomorphisms induces a \emph{strict} monoidal structure $\otimes \colon \mathcal{B}_{\otimes} \times \mathcal{B}_{\otimes}  \to \mathcal{B}_{\otimes}$, where $n \otimes m = n + m$.
Likewise, we have a symmetry $c_{n,m}$ on $\mathcal{B}_{\otimes}$, where $c_{n,m} \in \Aut{(A \otimes \K)^{\otimes (n+m)}}$ is the automorphism $(A \otimes \K)^{\otimes n} \otimes (A \otimes \K)^{\otimes m} \to (A \otimes \K)^{\otimes m} \otimes (A \otimes \K)^{\otimes n}$ switching the two factors. With these definitions $\mathcal{B}_{\otimes}$ becomes a permutative category. Define
\[
	B\Endo{A \otimes \K}^{\times}_{\otimes}: = \lvert N_{\bullet} \mathcal{B}_{\otimes} \rvert\ .
\]

\begin{lemma}\label{lem:B_Btensor_equiv}
The inclusion functor $J:\mathcal{B} \to \mathcal{B}_{\otimes}$ induces a homotopy equivalence of the corresponding classifying spaces $B\Aut{A \otimes \K} \to B\Endo{A \otimes \K}^{\times}_{\otimes}$. 	 Given an isomorphism $\psi \colon A \otimes \K \to (A \otimes \K)^{\otimes 2}$, the diagram
	\begin{equation}\label{eqn:multdiagram}
		\xymatrix{
			\mathcal{B} \times \mathcal{B} \ar[r]^-{\otimes_{\psi}} \ar[d]_-{J \times J} & \mathcal{B} \ar[d]^-{J} \\
			\mathcal{B}_{\otimes} \times \mathcal{B}_{\otimes} \ar[r]^-{\otimes} & \mathcal{B}_{\otimes}
		}
	\end{equation}
	commutes up to a natural transformation. In particular, the $H$-space structure of $B\Aut{A \otimes \K}$ agrees with the one on $B\Endo{A \otimes \K}_{\otimes}^{\times}$ up to homotopy.
\end{lemma}

\begin{proof}\label{pf:B_Btensor_equiv}
To prove the first statement we will construct auxiliary categories $\mathcal{E}$, $\mathcal{H}$ and $\mathcal{B}_{\otimes}^1$ together with inclusion functors $\mathcal{B} \to \mathcal{E} \to \mathcal{H} \to \mathcal{B}_{\otimes}^1 \to \mathcal{B}_{\otimes}$ that give a factorization of $J$. We then show that each of these functors induces a homotopy equivalence on classifying spaces. We will use the following two facts.

(a) Given two topological categories $\mathcal{C}$ and $\mathcal{D}$ together with continuous functors $F \colon \mathcal{C} \to \mathcal{D}$, $G \colon \mathcal{D} \to \mathcal{C}$ and natural transformations $F \circ G \Rightarrow \id{\mathcal{D}}$, $G \circ F \Rightarrow \id{\mathcal{C}}$, it follows that $F$ and $G$ induce a homotopy equivalence of the corresponding classifying spaces. This is a corollary of \cite[Prop.2.1]{Segal:paper-Classifying-spaces}.

(b) Consider two good simplicial spaces $X_{\bullet}$ and $Y_{\bullet}$ (``good'' refers to \cite[Definition A.4]{paper:SegalCatAndCoh}) together with a simplicial map $f_{\bullet} \colon X_{\bullet} \to Y_{\bullet}$. If $f_n \colon X_n \to Y_n$ is a homotopy equivalence for each $n \in \N_0$, then $\lvert f_{\bullet} \rvert \colon \lvert X_{\bullet} \rvert \to \lvert Y_{\bullet} \rvert$ is also a homotopy equivalence. This is proven in \cite[Proposition A.1 (ii) and (iv)]{paper:SegalCatAndCoh}. Note in particular, that the nerve $N_{\bullet}\mathcal{C}$ of a topological category $\mathcal{C}$ is good, if the map $\obj(\mathcal{C}) \to \mor(\mathcal{C})$, which sends an object to the identity on it, is a cofibration. This holds for all categories in this proof by Proposition \ref{prop:wellpointed}.

The object space of $\mathcal{E}$ consists of a single point and its morphism space is $\Endo{A \otimes \K}^{\times}$. From Lemma \ref{lem:full-AutAKfibration} and Theorem \ref{thm:stabContractible} we obtain that $\Aut{A \otimes \K} \to \Endo{A \otimes \K}^{\times}$ is a homotopy equivalence. Thus each component   $N_k\mathcal{B} \to N_k\mathcal{E}$  of the simplicial map $ N_{\bullet}\mathcal{B} \to  N_{\bullet}\mathcal{E}$ induced by the
 inclusion functor $\mathcal{B} \to \mathcal{E}$ is  a homotopy equivalence of spaces.    This yields a homotopy equivalence $\lvert N_{\bullet}\mathcal{B}\rvert \to \lvert N_{\bullet}\mathcal{E}\rvert$ by (b) above.

The category $\mathcal{B}^{1}_{\otimes}$ is the full subcategory of $\mathcal{B}_{\otimes}$ containing the objects $0$ and $1$. To see that the inclusion functor $\iota \colon \mathcal{B}^{1}_{\otimes} \to \mathcal{B}_{\otimes}$ is an equivalence of categories, we argue as follows: Define an inverse functor $\tau \colon \mathcal{B}_{\otimes} \to \mathcal{B}^{1}_{\otimes}$ for $\iota$, such that $\tau(m) = \min\{m,1\}$ on objects. Let  $\psi_0 = \id{\C}$ and fix isomorphisms $\psi_k \colon A \otimes \K \to (A \otimes \K)^{\otimes k}$ for $k \in \N$ with $\psi_1 = \id{A \otimes \K}$. Define $\tau(\beta) = \psi_{\ell}^{-1} \circ \beta \circ \psi_k$ for $\beta \in \hom_{\mathcal{B}_{\otimes}}(k,\ell)$. We have $\tau \circ \iota = \id{{\mathcal{B}_{\otimes}^{1}}}$ and the $\psi_k$ yield a natural transformation $\iota \circ \tau \Rightarrow \id{\mathcal{B}_{\otimes}}$. Thus, the map $\lvert N_{\bullet} \mathcal{B}_{\otimes}^{1}\rvert \to \lvert N_{\bullet} \mathcal{B}_{\otimes} \rvert$ induced by $\iota$ is a homotopy equivalence by (a).

Let $\mathcal{H}$ be the topological category with object space $\{0,1\}$ and morphism spaces $\hom_{\mathcal{H}}(0,0) = \{\id{A \otimes \K}\}$, $\hom_{\mathcal{H}}(0,1) = \hom_{\mathcal{H}}(1,1) = \Endo{A \otimes \K}^{\times}$ and $\hom(1,0) = \emptyset$. The composition is induced by the composition in $\Endo{A \otimes \K}^{\times}$. Note that there is a restriction functor $\mathcal{H} \to \mathcal{B}_{\otimes}^{1}$, which takes $\beta \in \hom_{\mathcal{H}}(0,1) = \Endo{A \otimes \K}^{\times}$ to $\widetilde{\beta} \in \hom_{\mathcal{B}_{\otimes}^{1}}(0,1) = \hom_{\mathcal{B}_{\otimes}}(0,1)$, where $\widetilde{\beta}(\lambda) = \lambda\beta(1 \otimes e)$ for $\lambda \in \C$. It maps the spaces $\hom_{\mathcal{H}}(0,0)$ and $\hom_{\mathcal{H}}(1,1)$ identically onto $\hom_{\mathcal{B}_{\otimes}^1}(0,0)$ and $\hom_{\mathcal{B}_{\otimes}^1}(1,1)$ respectively. By Lemma \ref{lem:full-AutAKfibration} and Theorem \ref{thm:stabContractible} the restriction map $\Endo{A \otimes \K}^{\times} \to \mathrm{Hom}(\C, A \otimes \K)^\times \cong \Proj{A \otimes \K}^{\times}$ is a homotopy equivalence. Therefore the simplicial map $N_{k}\mathcal{H} \to N_{k}\mathcal{B}_{\otimes}^{1}$ is a homotopy equivalence
   for each $k$, and hence $\lvert N_{\bullet}\mathcal{H}\rvert \to \lvert N_{\bullet}\mathcal{B}_{\otimes}^{1} \rvert$ is a homotopy equivalence by (b).

Let $\iota_{\mathcal{E}} \colon \mathcal{E} \to \mathcal{H}$ be the inclusion functor. Let $\tau_{\mathcal{E}}\colon \mathcal{H} \to \mathcal{E}$ be the functor, which maps the two objects of $\mathcal{H}$ to the one of $\mathcal{E}$ and which embeds the spaces $\hom_{\mathcal{H}}(0,0)$, $\hom_{\mathcal{H}}(0,1)$ and $\hom_{\mathcal{H}}(1,1)$ into $\Endo{A \otimes \K}^{\times}$ in a canonical way. We have $\tau_{\mathcal{E}} \circ \iota_{\mathcal{E}} = \id{\mathcal{E}}$. There is a natural transformation $\kappa \colon \id{\mathcal{H}} \Rightarrow \iota_{\mathcal{E}} \circ \tau_{\mathcal{E}}$ with $\kappa_1 = \id{A \otimes \K} \in \hom_{\mathcal{H}}(1,1)$ and $\kappa_0 = \id{A \otimes \K} \in \hom_{\mathcal{H}}(0,1)$. It follows that $\iota_{\mathcal{E}}$ also induces an equivalence on classifying spaces by (a). This concludes the proof of the first statement.

Let $\beta_1, \beta_2$ be morphisms in $\mathcal{B}$, then $(J \circ \otimes_{\psi})(\beta_1, \beta_2) = \beta_1 \otimes_{\psi} \beta_2 = \psi^{-1} \circ (\beta_1 \otimes \beta_2) \circ \psi \in \hom_{\mathcal{B}_{\otimes}}(1,1)$, whereas $\otimes \circ (J \times J)(\beta_1, \beta_2) = \beta_1 \otimes \beta_2 \in \hom_{\mathcal{B}_{\otimes}}(2,2)$ and $\psi \in \hom_{\mathcal{B}_{\otimes}}(1,2)$ provides a natural transformation $J \circ \otimes_{\psi} \Rightarrow \otimes \circ (J \times J)$. Thus these two functors induce
homotopic maps of classifying spaces by \cite[Prop.2.1]{Segal:paper-Classifying-spaces}.
  This completes the proof.
\end{proof}

\begin{corollary} \label{cor:AutIsInfLoopSpace}
The space $B\Aut{A \otimes \K}$ inherits an infinite loop space structure via the homotopy equivalence $B\Aut{A \otimes \K} \to B\Endo{A \otimes \K}^{\times}_{\otimes}$ in such a way that the induced $H$-space structure of $B\Aut{A \otimes \K}$ agrees with the one given by $\mu_{\psi}$.
\end{corollary}

\begin{proof} \label{pf:AutIsInfLoopSpace}
By Theorem \ref{thm:InfLoopSpace}, $B\Endo{A \otimes \K}^{\times}_{\otimes}$ is an infinite loop space with $H$-space structure induced by the tensor product of $\mathcal{B}_{\otimes}$. By Lemma \ref{lem:B_Btensor_equiv}, $B\Aut{A \otimes \K} \to B\Endo{A \otimes \K}^{\times}_{\otimes}$ is a homotopy equivalence and a map of $H$-spaces.
\end{proof}
	
\begin{theorem}\label{thm:MainTheorem} Let $A$ be a strongly self-absorbing C*-algebra.
\begin{enumerate}[(a)]
	\item The monoid $(\BunX{A \otimes \K}, \otimes)$ of isomorphism classes of principal $\Aut{A \otimes \K}$-bundles is an abelian group.
	\item $B\Aut{A \otimes \K}$ is the first space in a spectrum defining a cohomology theory $E^{\bullet}_{A}$ with
\(
	E_{A}^{0}(X) = [X, \Aut{A \otimes \K}]\)  and \(E_{A}^{1}(X) = \BunX{A \otimes \K}.
\)
    \item If $X$ is a compact metrizable space  and $A\neq \C$, then $E_{A}^{0}(X)  \cong K_0(C(X)\otimes A)^\times_{+}$.
\end{enumerate}	
\end{theorem}

\begin{proof}\label{pf:MainTheorem}
By Corollary \ref{cor:AutIsInfLoopSpace} the space $B\Aut{A \otimes \K}$ is an infinite loop space with $H$-space structure given by with $\otimes_{\psi}$, which implies the first part. As described above, an infinite loop space yields a spectrum and therefore a cohomology theory via iterated delooping. If we consider $B\Aut{A \otimes \K}$ as the first space of the spectrum, we obtain the $0$th one by forming the loop space. But this is
\[
	\Omega B\Aut{A \otimes \K} \simeq \Aut{A \otimes \K}\ ,
\] 	
which proves the second statement. The last statement follows from Theorem~\ref{thm:MapsEndAK}.
\end{proof}

\begin{corollary}\label{cou:uAutIsInfLoopSpace}
For any strongly self-absorbing C*-algebra $A$ the space $B\uAut{A \otimes \K}$ is an infinite loop space with respect to its natural tensor product operation. The corresponding cohomology theory is denoted by $\bar{E}^*_A(X)$.
\end{corollary}

\begin{proof}
The proof is entirely similar to the proof of Theorem~\ref{thm:MainTheorem}, except that we replace the category $\mathcal{B}$ by the topological category $\mathcal{B}^0$  which has as its object space just a single point and the group $\uAut{A \otimes \K}$ as its morphism space. Likewise we replace the category  $\mathcal{B}_{\otimes}$ by the category $\mathcal{B}^0_{\otimes}$ defined as follows. The object space of $\mathcal{B}^0_{\otimes}$ is $\N_0$. The morphisms $\hom(n,m)$ are given by those maps $\alpha$ in $\hom((A \otimes \K)^{\otimes n}, (A \otimes \K)^{\otimes m})$ with the property that $[\alpha((1\otimes e)^{\otimes n})]=[(1\otimes e)^{\otimes m}]$ in $K_0((A \otimes \K)^{\otimes m})$. The proof of lemma \ref{lem:B_Btensor_equiv} still works with the following modifications: There are straightforward replacements $\mathcal{E}^0$, $\left(\mathcal{B}_{\otimes}^{1}\right)^0$ and $\mathcal{H}^0$ of the categories $\mathcal{E}$, $\mathcal{B}_{\otimes}^{1}$ and $\mathcal{H}$ by taking those endomorphisms that preserve the $K$-theory class of $1 \otimes e$. The isomorphisms $\psi_k$ used in the proof can be chosen such that $\psi_k(1\otimes e)=(1\otimes e)^{\otimes k}$. The restriction functor $\mathcal{H}^0 \to \left(\mathcal{B}_{\otimes}^{1}\right)^0$ still induces an equivalence by lemma \ref{lem:AutAKfibration} and theorem \ref{thm:stabContractible}.
\end{proof}

\begin{remark}\label{rem:continuity}
We have seen that the classifying space $B\Aut{A \otimes \K}$ has the homotopy type of a CW complex. Since its homotopy groups are countable, it follows that this space is homotopy equivalent to a locally finite simplicial complex and hence to an absolute neighborhood extensor, see \cite[Thm.6.1, p.137]{book:LundellWeingram}. It follows that $E_{A}^{1}(X)$ is a continuous functor in the sense that if $X$ is the projective limit of projective system $(X_n)_n$ of compact metrizable spaces, then $E_{A}^{1}(X)\cong \varinjlim E_{A}^{1}(X_n)$, see \cite[Thm.11.9, p.287]{Book:Eilenberg-Steenrod}. Since any compact metrizable space $X$ is the projective limit of a system of finite polyhedra $(X_n)_n$ by \cite{Book:Eilenberg-Steenrod}, one can approach the computation of $E_{A}^{1}(X)$ by first computing $E_{A}^{1}(X_n)$ using the Atiyah-Hirzebruch spectral sequence.
\end{remark}

\section{A generalized Dixmier-Douady theory}\label{section:DD}
Recall from \cite[10.4]{book:Dixmier:C*-algebras} that if $\mathcal{B}=((B(x))_{x\in X},\Theta)$  is a continuous field
of C*-algebras over a locally compact space $X$, the C*-algebra $B$ associated to $\mathcal{B}$ consists
of all elements $\theta$ of $\Theta$ such that the function $x\mapsto \|\theta(x)\|$ vanishes at infinity. 
As it has become customary in the literature, the C*-algebra $B$ will be also referred to as a continuous field
of C*-algebras. Note that $B=\Theta$ if $X$ is compact.
\begin{definition}\label{def:Fell}
Let $B$ be a continuous field of C*-algebras over a locally compact metrizable space $X$ whose fibers are stably isomorphic to strongly self-absorbing C*-algebras, which are not necessarily mutually isomorphic. We say that
$B$ satisfies \emph{the Fell condition} if for each point
$x\in X$, there is a closed neighborhood $V$ of $x$ and a projection $p\in B(V)$ such that
$[p(v)]\in K_0(B(v))^{\times}$  for all $v\in V$. If one can choose $V=X$, then we say that $B$ satisfies the global Fell
condition.

\end{definition}
\begin{theorem}\label{thm:trivialization} Let $A$ be a strongly self-absorbing C*-algebra. Let $X$ be a  locally compact
space of finite covering dimension and let $B$ be a separable continuous field of C*-algebras over $X$ with all fibers abstractly isomorphic to $A\otimes \K$.
Then $B$ is locally trivial if and only if it satisfies Fell's condition. If $X$ is compact, then $B$ is trivial if and only if
$B$ satisfies the global Fell condition.
\end{theorem}
\begin{proof}  Suppose that there is a projection
$p\in B(V)$ such that $[p(v)]\in K_0(B(v))^{\times}$  for all $v$ in a compact subset $V$ of $X$.
 We will show that $B(V)\cong C(V)\otimes A \otimes \K$. First we observe that by Lemma~\ref{lemma:p-tensor-q} it follows that
 $p(v)B(v)p(v)\cong A$, since $B(v)\cong A\otimes \K$. Therefore
 $pB(V)p$ is a unital continuous field over a finite dimensional space with fibers isomorphic to $A$ and hence $pB(V)p\cong C(V)\otimes A$
by \cite{paper:Dadarlat-Winter-fields}. Second, since $p$ is a full projection, we have that
$pB(V) p\otimes \K \cong B(V)\otimes \K$ as $C(V)$-algebras by \cite{paperL.G.Brown.Stable.Isom}. Third, $B(V)\otimes \K\cong B(V)$  by \cite{paper.Hirshberg.Rordam.Winter} since $V$ is finite dimensional
and each fiber of $B$ is stable. Putting these facts together we obtain the desired conclusion:
\[B(V)\cong B(V)\otimes \K \cong pB(V) p\otimes \K \cong C(V)\otimes A \otimes \K.\qedhere \]
\end{proof}
\begin{corollary}\label{cor-automatic-triviality}
 Let $X$ be a  locally compact
space of finite covering dimension.  Any separable continuous field of C*-algebras over $X$ with all fibers abstractly isomorphic to $M_{\Q}\otimes \K$ is locally trivial.
\end{corollary}
\begin{proof} Let $B$ be a continuous field as in the statement. In view of Theorem~\ref{thm:trivialization} it suffices to show
that $B$ satisfies the Fell condition. Fix $x\in X$ and let $p_0\in B(x)\cong M_{\Q}\otimes \K$ be a non-zero projection. Since $\C$ is semiprojective, we can lift $p_0$
to a projection in $A(V)$ for some closed neighborhood $V$ of $x$. Since $A$ is a continuous field, the map {$v\mapsto \|p(v)\|$} is continuous.
Thus by shrinking $V$ we can arrange that $p(v)\neq 0$ for all $v\in V$, since $\|p(x)\|=\|p_0\|=1$. Since $B(v) \cong M_\Q\otimes \K$ it follows that $[p(v)]\in K_0(M_{\Q})\setminus \{0\}\cong \Q^\times \cong K_0(M_\Q)^{\times}$.
\end{proof}
Having obtained an efficient criterion for local triviality, we now turn to the question of classifying locally trivial continuous fields of C*-algebras
by cohomological invariants.
Let $X$ be a finite connected CW complex.
Let $R=K_0(A)$ and let $R_+^\times$ denote the multiplicative abelian group $K_0(A)^\times_+$.
If $A$ is purely infinite, then $K_0(A)^\times_+=K_0(A)^\times$ and so $R_+^\times=R^\times$.
Suppose that $A$ satisfies the UCT. Then $K_1(A)=0$ by \cite{paper:TomsWinter}.

 The coefficients of the generalized  cohomology theory $E_A^*(X)$ were computed in Theorem~\ref{thm:homotopy-groups}. Consequently,
by \cite{Hilton:book-generalized-cohomology}, the $E_2$-page of the Atiyah-Hirzebruch spectral sequence for the generalized cohomology $E_A^*(X)$, $A\neq \C$,
  looks as shown below.
\begin{center}
\begin{tikzpicture}
  \matrix (m) [matrix of math nodes,
    nodes in empty cells,nodes={minimum width=5ex,
    minimum height=5ex,outer sep=-3pt},
    column sep=1ex,row sep=1ex]{
         \quad\strut   &   0  &  1  &  2  &  3  & \\
         0     &   H^0(X,R_+^\times)  &  H^1(X,R_+^\times)  &  H^2(X,R_+^\times) &  H^3(X,R_+^\times)  &\\
         -1     &   0  &  0  &  0  &  0  &\\
         -2     &   H^0(X,R)  &  H^1(X,R)  &  H^2(X,R)  & H^3(X,R) \\
         -3     &   0  &  0  &  0  &  0  &\\
         -4  \strut   &   H^0(X,R)  &  H^1(X,R)  &  H^2(X,R)  & H^3(X,R) \\};
\draw[thick] (m-1-1.east) -- (m-6-1.east) ;
\draw[thick] (m-1-1.south) -- (m-1-6.south) ;
\end{tikzpicture}
\end{center}
If $A=\C$, all the rows
of the $E_2$-page of $E^1_\mathbb{C}(X)$ are null with the exception of the $(-2)$-row
whose entries are $H^p(X,\Z)$, $p\geq 0$.
Since  the differentials in the Atiyah-Hirzebruch spectral sequence are torsion operators,
 \cite[Thm.2.7]{paper:Arlettaz}, we obtain the following.
\begin{corollary}\label{cor:computationEA-0} Let $X$ be a finite connected CW complex such that $H^*(X,R)$ is
torsion free. If $A\neq \C$ satisfies the UCT, then
\[
 \BunX{A \otimes \K}\cong E_{A}^{1}(X) \cong H^1(X,R_+^\times)\times \prod_{k \geq 1} H^{2k+1}(X,R)\ .
\]
\end{corollary}
 \begin{corollary}\label{cor:computationEA-1} Let $X$ be a compact  metrizable space and let $M_{\mathbb{Q}}$ denote the universal UHF-algebra
with $K_0(M_{\mathbb{Q}})\cong \mathbb{Q}$. Then there are natural
isomorphism of groups
\[
\BunX{M_{\mathbb{Q}} \otimes \K}\cong E_{M_{\mathbb{Q}}}^{1}(X)  \cong H^1(X,\mathbb{Q}_+^\times)\oplus \bigoplus_{k \geq 1} H^{2k+1}(X,\mathbb{Q})\ .
\]
\[
\BunX{M_{\mathbb{Q}} \otimes\mathcal{O}_\infty\otimes \K}\cong E_{M_{\mathbb{Q}}\otimes \mathcal{O}_\infty}^{1}(X)  \cong H^1(X,\mathbb{Q}^\times)\oplus \bigoplus_{k \geq 1} H^{2k+1}(X,\mathbb{Q})\ .
\]
\end{corollary}
\begin{proof} Set $h^*(X)=E^*_A(X)$, $\bar{h}^*(X)=\bar{E}^*_A(X)$ (see Cor.~\ref{cou:uAutIsInfLoopSpace}) and $R=K_0(A)$ where $A$ is either $M_{\mathbb{Q}}$ or $M_{\mathbb{Q}}\otimes \mathcal{O}_\infty$. We will show that there are natural isomorphisms
(i) $h^1(X)\cong H^{1}(X,R^\times_{+})\oplus \bar{h}^1(X)$ and (ii)  $\bar{h}^1(X)\cong \bigoplus_{k \geq 1} H^{2k+1}(X,\mathbb{Q})$. Note that $\bar{h}^*(pt)=t\,\mathbb{Q}[t]$ with $deg(t)=-2$ and $h^*(pt)=R^{\times}_{+}\oplus t\,\mathbb{Q}[t]$.
Suppose first that $X$ is a finite connected CW-complex.
Then (ii) follows
 by applying  the isomorphism established in equation $(3.20)$ of \cite[p.48]{Hilton:book-generalized-cohomology} since  $\bar{h}^*(pt)$ is a vector spaces over $\mathbb{Q}$.
  If $G$ is a topological group and $H$ a normal subgroup of $G$ such that
$H\to G \to G/H$  is a principal $H$-bundle,    then there is a homotopy fibre sequence
$G\to G/H \to BH \to BG \to B(G/H)$ and hence an exact sequence of pointed sets
$[X,G]\to [X,G/H] \to [X,BH] \to [X,BG] \to [X,B(G/H)]$.
Using this for the principal bundle from Corollary~\ref{cor:101} we obtain an exact
sequence of groups $0\to \bar{h}^1(X)\to h^1(X)\stackrel{\delta_0}\longrightarrow H^1(X,R^\times_{+})$. We want to compare this sequence with the exact sequence $0 \to F^2h^1(X) \to h^1(X)  \to H^1(X,R^\times_{+})\to 0$ given by
 the Atiyah-Hirzebruch spectral sequence.
Recall that $F^2h^1(X)=\mathrm{ker}(h^1(X)\to h^1(X_1))$, where  $X_1$ is the 1-skeleton of $X$.
 Since both maps with target $H^1(X_1,R^\times_{+})$
are injective in the following commutative diagram induced by $X_1\hookrightarrow X$
\[
	\xymatrix{
		h^1(X) \ar[r]^-{\delta_0}\ar[d]  &H^1(X,R^\times_{+})\ar[d]\\
	h^1(X_1) \ar[r] &H^1(X_1,R^\times_{+})
}\]
 we deduce that $F^2h^1(X)\cong\mathrm{ker}(\delta_0)\cong \bar{h}^1(X)$
and hence obtain an exact sequence $0\to \bar{h}^1(X)\to h^1(X)\to H^1(X,R^\times_{+})\to 0.$
Since $\bar{h}^1(X)$ is a divisible
group it follows that $h^1(X)$ splits as $H^1(X,R^\times_{+})\oplus \bar{h}^1(X)$. To verify that there is a natural
splitting one employs the natural transformation $h^*(X)\to \bar{h}^*(X)$ induced by the coefficient map
$h^*(pt)\to \bar{h}^*(pt)$, $(r,f(t))\mapsto f(t)$, see \cite[Thm.3.22]{Hilton:book-generalized-cohomology}.

For the general case we write $X$ as a projective limit of a system of  polyhedra $(X_n)_n$
and then we apply the continuity property of $E_{A}^{1}(X)$ as discussed in Remark~\ref{rem:continuity}.
\end{proof}

Let $A\ncong \OO_2$ be a strongly self-absorbing C*-algebra that satisfies the UCT.  Then $A\otimes \QQ \otimes\OO_\infty\cong \QQ \otimes\OO_\infty$ by \cite{paper:TomsWinter}.
The canonical unital embedding  $A \to A\otimes \QQ \otimes\OO_\infty\cong \QQ \otimes\OO_\infty$ induces a morphism of groups
$\Aut{A\otimes \K}\to \Aut{ \QQ \otimes\OO_\infty\otimes \K}$ and hence a morphism of groups
\[\delta: E_A^1(X)\to E^1_{\QQ \otimes\OO_\infty}(X)\cong H^1(X,\mathbb{Q}^\times)\oplus \bigoplus_{k \geq 1} H^{2k+1}(X,\mathbb{Q}).\]
\begin{definition}
We define rational characteristic classes $\delta_0:\BunX{A \otimes \K}\to H^1(X,\mathbb{Q}^\times)$ and
\(\delta_k:\BunX{A \otimes \K} \to H^{2k+1}(X,\Q)\), $k\geq 1$, to be the components of the map $\delta$ from above. It is clear that $\delta_0$ is lifts to a map $\delta_0: \BunX{A \otimes \K}\to H^1(X,K_0(A)_+^\times)$  induced by the morphism of groups $\Aut{A\otimes \K}\to \pi_0(\Aut{A\otimes \K})\cong K_0(A)_+^\times$ and which gives the obstruction to reducing the structure group to $\uAut{A\otimes \K}$.
We will see in
Corollary~\ref{cor:computationEA-2}
that $\delta_1$ also lifts to an integral class with values in $H^3(X,\Z)$ for $A = \ZZ$.
 One has  $\delta_k(B_1\otimes B_2)=\delta_k(B_1)+\delta_k(B_2)$, $k\geq 0$.
\end{definition}
Since the  differentials in the Atiyah-Hirzebruch spectral sequence are torsion operators we deduce:
\begin{corollary}\label{cor:computationEA-ZZ}
Let $A$ be a strongly self-absorbing C*-algebra that satisfies the UCT.
Let $X$ be a finite connected CW complex such that $H^*(X,\Z)$ is torsion free. Then:

(i)  $B_1, B_2\in \BunX{A\otimes \K}$ are isomorphic if and only
$\delta_k(B_1)=\delta_k(B_2)$ for all $k\geq 0$.

(ii)  $\BunX{\ZZ \otimes \K}\cong \bigoplus_{k \geq 1} H^{2k+1}(X,\Z)$ and $\BunX{\mathcal{O}_\infty \otimes \K}\cong H^1(X,\Z/2)\oplus \bigoplus_{k \geq 1} H^{2k+1}(X,\Z)$.
\end{corollary}

\begin{corollary}\label{cor:computationEA-2} 
Let $X$ be a compact connected metrizable space. Let $A$ be a strongly self-absorbing $C^*$-algebra, which satisfies the UCT. Then $ H^{i+2}(X, K_0(A))$ is a natural direct summand of $\bar{E}^i_{A}(X)$ 
for all $i\geq 0$. It follows that there is a natural homomorphism $$\bar{\delta}_1 \colon \bar{E}^1_A(X) \to H^3(X,K_0(A))$$ giving back the usual Dixmier-Douady class for $A = \C$.
\end{corollary}
\begin{proof} Recall that ${E}^{i}_\mathbb{C}(X)\cong \bar{E}^i_\mathbb{C}(X)\cong H^{i+2}(X,\Z)$.
Using the continuity properties discussed in Remark~\ref{rem:continuity}, we may assume that $X$ is a finite connected CW complex. Since $A$ satisfies the UCT, $K_1(A)= 0$ and $K_0(A) \subset \Q$ is flat and satisfies $K_0(A) \otimes K_0(A) \cong K_0(A)$. The  natural transformation  of cohomology theories $\bar{E}^*_A(X) \stackrel{\cong}\longrightarrow \bar{E}^*_A(X) \otimes K_0(A)$ is an isomorphism since it is so on coefficients. The unital map  $\C \to  A$ induces a natural
transformation of cohomology theories $T \colon \bar{E}^*_{\C}(X) \otimes K_0(A) \to \bar{E}^*_{A}(X) \otimes K_0(A) \cong \bar{E}^*_A(X)$. The desired conclusion follows now from the naturality of the Atiyah-Hirzebruch spectral sequence since all the rows
of the $E_2$-page of $\bar{E}^*_\mathbb{C}(X) \otimes K_0(A)$ are null with the exception of the $(-2)$-row and $T$ induces the identity map on this row due to the isomorphism $\pi_2(\Aut{\K}) \otimes K_0(A) \to \pi_2(\Aut{A \otimes \K}) \otimes K_0(A) \cong \pi_2(\Aut{A \otimes \K})$, see Theorem~\ref{thm:homotopy-groups}. The edge homomorphism $\bar{E}^i_{A}(X) \to H^{i+2}(X, K_0(A))$ and $T$ give the splitting.
\end{proof}

\begin{corollary}\label{cor:103} Let $X$ be a compact metrizable space and let $A$ be a strongly self-absorbing C*-algebra. Two bundles $B_1,B_2\in \BunX{A\otimes \K}$ are isomorphic if and only if $B_1\otimes \OO_\infty \cong B_2 \otimes \OO_\infty$.
\end{corollary}
\begin{proof} Without any loss of generality we may assume that $X$ is a finite   CW-complex.
 By Corollary~\ref{cor:101} there is a  commutative diagram with exact rows
\[
	\xymatrix{0\ar[r]&
		\mathrm{Aut}_{0}(A\otimes \K) \ar[r]\ar[d] & \Aut{A \otimes \K} \ar[d]\ar[r] &K_0(A)^{\times}_{+}\ar[d]\ar[r]&0\\
	0\ar[r]&	\mathrm{Aut}_{0}(\OO_\infty\otimes A\otimes \K) \ar[r] & \Aut{\OO_\infty\otimes A \otimes \K}
\ar[r] &K_0(A)^{\times}	\ar[r]&0}\]
Passing to classifying spaces  we obtain a commutative diagram:
\[
	\xymatrix{0\ar[r]&
		[X,B\mathrm{Aut}_{0}(A\otimes \K)] \ar[r]\ar[d]_{i} & [X,B\Aut{A \otimes \K} ]\ar[d]_{T}\ar[r] &{H}^1(X,K_0(A)^{\times}_{+})\ar[d]_{j}\\
	0\ar[r]&	[X,B\mathrm{Aut}_{0}(\OO_\infty\otimes A\otimes \K)] \ar[r] & [X,B\Aut{\OO_\infty\otimes A \otimes \K}]
\ar[r] & {H}^1(X,K_0(A)^{\times}	)}
\]
This is a diagram of abelian groups by Theorem~\ref{thm:MainTheorem} and Corollary~\ref{cou:uAutIsInfLoopSpace}.
The map $j$ is injective. This follows from the exactness
of the sequence ${H}^0(X,R_2)\to{H}^0(X,R_2/R_1)\to {H}^1(X,R_1)\to {H}^1(X,R_2)$ induced by an inclusion
of discrete abelian groups $R_1\hookrightarrow R_2$. Let us argue that the map $i$ is also injective.
 If $A\neq \C$,  this follows from Corollary~\ref{cor:102}, whereas for $A=\C$  this was proved in Corollary~\ref{cor:computationEA-2}. The five lemma implies now that the map
 $T:E^1_{A}(X)\to E^1_{A\otimes \OO_\infty}(X)$ is injective.
\end{proof}

Finally we address the question to what extent the results (i) and (ii) of Dixmier and Douady mentioned in the introduction admit generalizations to our context. The following statement corresponds to  (i)  and the first part of (ii).
\begin{corollary}\label{cor:D} Let $B$ be a separable continuous field of C*-algebras over a compact metrizable space whose fibers are Morita equivalent to the same strongly self-absorbing C*-algebra $A$. Suppose  that $B$ satisfies Fell's condition. Then for each $x\in X$, there is a closed neighborhood $V$ of $x$ with the following property. There exists a unital separable continuous field $D$ over $V$ with fibers isomorphic to $A$
 such that $B(V)\otimes \K\cong D \otimes \K$. If  $A$ is finite dimensional or if $X$ is finite dimensional,
then $B\otimes \K$ is locally trivial and therefore we can associate to it an invariant $\delta(B)\in E^1_A(X)$
which classifies $B\otimes \K$ up to isomorphism of continuous fields, and $B$ up to Morita equivalence over $X$.
\end{corollary}
\begin{proof} Let $x\in X$. Let $p$ and $V$ be as in Definition~\ref{def:Fell}. Letting $D:=pB(V)p$ we have already seen in the proof of Theorem~\ref{thm:trivialization} that $B(V)\otimes \K\cong D \otimes \K$ and that all the fibers of $D$ are isomorphic to $A$.
If $A$ is finite dimensional, then $A=\mathbb{C}$, and so obviously $D=C(V)$.  This corresponds to the result (i)
of Dixmier and Douady.
If $X$ is finite dimensional, then
  $B(V)\otimes \K\cong C(V)\otimes A \otimes \K$  by Theorem~\ref{thm:trivialization}.
  We conclude the proof by applying Theorem~\ref{thm:MainTheorem}.
\end{proof}
As we have just seen, the  class $\delta(B)\in E^1_A(X)$ is defined for continuous fields with fibers Morita equivalent to $A$
which satisfy Fell's condition.  Furthermore, one can associate rational characteristic classes to certain continuous fields
which are typically very far from being locally trivial and whose fibers are not necessarily Morita equivalent to each other.
\begin{corollary}\label{cor:char-class} Let $B$ be a separable continuous field of C*-algebras over a finite dimensional compact metrizable space whose fibers are Morita equivalent to (possibly different) strongly self-absorbing C*-algebras satisfying the UCT.
 Suppose that for each $x\in X$ there is a closed neighborhood $V$ of $x$ and a projection $p\in B(V)$ such that
 $[p(v)]\neq 0$ in $K_0(B(v))$ for all $v\in V$. Then $B_\sharp:=B\otimes \mathcal{O}_\infty\otimes M_{\mathbb{Q}}\otimes \K$
is locally trivial and so we can associate ``stable" rational characteristic classes to $B$, by defining $\delta_k^{stable}(B)=\delta_k(B_\sharp)\in H^{2k+1}(X,\mathbb{Q})$.
These cohomology classes determine $B_\sharp$ up to an isomorphism.
\end{corollary}
\begin{proof} The fibers  $B_\sharp(x)=B(x)\otimes \mathcal{O}_\infty\otimes M_{\mathbb{Q}}\otimes \K$
 satisfy the UCT and they are  stabilized strongly self-absorbing
Kirchberg algebras not isomorphic to $\OO_2\otimes \K$ since $K_0(B(x))\neq 0$. It follows by \cite{paper:TomsWinter} that they are all isomorphic to $\mathcal{O}_\infty\otimes M_{\mathbb{Q}}\otimes \K$ and moreover that the induced map
$K_0(B(x))\to K_0(B_\sharp(x))$ is injective for all $x\in X$. It follows that $B_\sharp$  satisfies
Fell's condition and hence it is locally trivial by Theorem~\ref{thm:trivialization}. We conclude the proof by applying
Corollary~\ref{cor:computationEA-1}.
\end{proof}

\renewcommand*{\bibfont}{\small}
\bibliographystyle{plain}

\begin{thebibliography}{10}

\bibitem{book:Adams}
John~Frank Adams.
\newblock {\em Infinite loop spaces}, volume~90 of {\em Annals of Mathematics
  Studies}.
\newblock Princeton University Press, Princeton, N.J., 1978.

\bibitem{paper:Arlettaz}
Dominique Arlettaz.
\newblock The order of the differentials in the {A}tiyah-{H}irzebruch spectral
  sequence.
\newblock {\em $K$-Theory}, 6(4):347--361, 1992.

\bibitem{paper:AtiyahSegal}
Michael Atiyah and Graeme Segal.
\newblock Twisted {$K$}-theory.
\newblock {\em Ukr. Mat. Visn.}, 1(3):287--330, 2004.

\bibitem{paper:AtiyahSegal-2}
Michael Atiyah and Graeme Segal.
\newblock Twisted {$K$}-theory and cohomology.
\newblock In {\em Inspired by {S}. {S}. {C}hern}, volume~11 of {\em Nankai
  Tracts Math.}, pages 5--43. World Sci. Publ., Hackensack, NJ, 2006.

\bibitem{Bellissard:Hall-effect}
Jean Bellissard.
\newblock {$K$}-theory of {$C^\ast$}-algebras in solid state physics.
\newblock In {\em Statistical mechanics and field theory: mathematical aspects
  ({G}roningen, 1985)}, volume 257 of {\em Lecture Notes in Phys.}, pages
  99--156. Springer, Berlin, 1986.

\bibitem{paper:Benson-Kumjian-Phillips}
David~J. Benson, Alex Kumjian, and N.~Christopher Phillips.
\newblock Symmetries of {K}irchberg algebras.
\newblock {\em Canad. Math. Bull.}, 46(4):509--528, 2003.

\bibitem{book:BlackadarK-theory}
B.~Blackadar.
\newblock {\em {$K$}-theory for operator algebras}, volume~5 of {\em
  Mathematical Sciences Research Institute Publications}.
\newblock Cambridge University Press, Cambridge, second edition, 1998.

\bibitem{book:BlackadarOpAlg}
B.~Blackadar.
\newblock {\em Operator algebras}, volume 122 of {\em Encyclopaedia of
  Mathematical Sciences}.
\newblock Springer-Verlag, Berlin, 2006.
\newblock Theory of $C{^{*}}$-algebras and von Neumann algebras, Operator
  Algebras and Non-commutative Geometry, III.

\bibitem{paper:BouwknegtMathai}
Peter Bouwknegt and Varghese Mathai.
\newblock D-branes, {$B$}-fields and twisted {$K$}-theory.
\newblock {\em J. High Energy Phys.}, (3):Paper 7, 11, 2000.

\bibitem{paperL.G.Brown.Stable.Isom}
Lawrence~G. Brown.
\newblock Stable isomorphism of hereditary subalgebras of {$C^*$}-algebras.
\newblock {\em Pacific J. Math.}, 71(2):335--348, 1977.

\bibitem{paper:BGR}
Lawrence~G. Brown, Philip Green, and Marc~A. Rieffel.
\newblock Stable isomorphism and strong {M}orita equivalence of
  {$C\sp*$}-algebras.
\newblock {\em Pacific J. Math.}, 71(2):349--363, 1977.

\bibitem{Connes:bookNG}
Alain Connes.
\newblock {\em Noncommutative geometry}.
\newblock Academic Press Inc., San Diego, CA, 1994.

\bibitem{paper:CuntzHigson}
Joachim Cuntz and Nigel Higson.
\newblock Kuiper's theorem for {H}ilbert modules.
\newblock In {\em Operator algebras and mathematical physics ({I}owa {C}ity,
  {I}owa, 1985)}, volume~62 of {\em Contemp. Math.}, pages 429--435. Amer.
  Math. Soc., Providence, RI, 1987.

\bibitem{paper:DadarlatPiAutAK}
Marius Dadarlat.
\newblock The homotopy groups of the automorphism group of {K}irchberg
  algebras.
\newblock {\em J. Noncommut. Geom.}, 1(1):113--139, 2007.

\bibitem{Dad:fiberwiseKK}
Marius Dadarlat.
\newblock Fiberwise {$KK$}-equivalence of continuous fields of
  {$C^*$}-algebras.
\newblock {\em J. K-Theory}, 3(2):205--219, 2009.

\bibitem{Dadarlat-Pennig:unit_spectra}
Marius Dadarlat and Ulrich Pennig
\newblock {Unit spectra of {$K$}-theory from strongly self-absorbing {$C^*$}-algebras}
\newblock{preprint 2013}

\bibitem{paper:Dadarlat-Winter-fields}
Marius Dadarlat and Wilhelm Winter.
\newblock Trivialization of {$C(X)$}-algebras with strongly self-absorbing
  fibres.
\newblock {\em Bull. Soc. Math. France}, 136(4):575--606, 2008.

\bibitem{paper:DadarlatKK1}
Marius Dadarlat and Wilhelm Winter.
\newblock On the {$KK$}-theory of strongly self-absorbing {$C^*$}-algebras.
\newblock {\em Math. Scand.}, 104(1):95--107, 2009.

\bibitem{paper:DixmierDouady}
Jacques Dixmier and Adrien Douady.
\newblock Champs continus d'espaces hilbertiens et de {$C^{\ast} $}-alg\`ebres.
\newblock {\em Bull. Soc. Math. France}, 91:227--284, 1963.

\bibitem{book:Dixmier:C*-algebras}
J.~Dixmier.
\newblock {\em {$C^*$-algebras}}.
\newblock North Holland, Amsterdam, 1982.

\bibitem{paper:DonovanKaroubi}
P.~Donovan and M.~Karoubi.
\newblock Graded {B}rauer groups and {$K$}-theory with local coefficients.
\newblock {\em Inst. Hautes \'Etudes Sci. Publ. Math.}, (38):5--25, 1970.

\bibitem{paper:Eckmann-Hilton}
B.~Eckmann and P.~J. Hilton.
\newblock Group-like structures in general categories. {I}. {M}ultiplications
  and comultiplications.
\newblock {\em Math. Ann.}, 145:227--255, 1961/1962.

\bibitem{Book:Eilenberg-Steenrod}
Samuel Eilenberg and Norman Steenrod.
\newblock {\em Foundations of algebraic topology}.
\newblock Princeton University Press, Princeton, New Jersey, 1952.

\bibitem{paper:Herman-Rosenberg}
Richard~H. Herman and Jonathan Rosenberg.
\newblock Norm-close group actions on {$C^{\ast} $}-algebras.
\newblock {\em J. Operator Theory}, 6(1):25--37, 1981.

\bibitem{Hilton:book-generalized-cohomology}
Peter Hilton.
\newblock {\em General cohomology theory and {$K$}-theory}.
\newblock { vol~1   London Mathematical Society Lecture Note Series}.
\newblock Cambridge University Press, London, 1971.

\bibitem{paper.Hirshberg.Rordam.Winter}
Ilan Hirshberg, Mikael R{\o}rdam, and Wilhelm Winter.
\newblock {$\mathcal C_0(X)$}-algebras, stability and strongly self-absorbing
  {$C^*$}-algebras.
\newblock {\em Math. Ann.}, 339(3):695--732, 2007.

\bibitem{paper:Huber}
Peter~J. Huber.
\newblock Homotopical cohomology and \v{C}ech cohomology.
\newblock {\em Math. Ann.}, 144:73--76, 1961.

\bibitem{paper:Jiang-nonstable}
Xinhui Jiang.
\newblock {Nonstable $K$-theory for {$\mathcal{Z}$}-stable C*-algebras}.
\newblock
  \href{http://front.math.ucdavis.edu/9707.5228}{arXiv:9707.5228}[math.OA],
  1997.

\bibitem{paper:Kodaka}
Kazunori Kodaka.
\newblock Full projections, equivalence bimodules and automorphisms of stable
  algebras of unital {$C^*$}-algebras.
\newblock {\em J. Operator Theory}, 37(2):357--369, 1997.

\bibitem{Landsman:book}
N.~P. Landsman.
\newblock {\em Mathematical topics between classical and quantum mechanics}.
\newblock Springer Monographs in Mathematics. Springer-Verlag, New York, 1998.

\bibitem{book:LundellWeingram}
A.T. Lundell and S.~Weingram.
\newblock {\em {The topology of CW complexes.}}
\newblock {The University Series in Higher Mathematics. New York etc.: Van
  Nostrand Reinhold Company. VIII, 216 p. }, 1969.

\bibitem{paper:May}
J.~P. May.
\newblock {\em The geometry of iterated loop spaces}.
\newblock Springer-Verlag, Berlin, 1972.
\newblock Lectures Notes in Mathematics, Vol. 271.

\bibitem{paper:MayGroupCompletion}
J.~P. May.
\newblock {$E_{\infty }$} spaces, group completions, and permutative
  categories.
\newblock In {\em New developments in topology ({P}roc. {S}ympos. {A}lgebraic
  {T}opology, {O}xford, 1972)}, pages 61--93. London Math. Soc. Lecture Note
  Ser., No. 11. Cambridge Univ. Press, London, 1974.

\bibitem{paper:MayPermutative}
J.~P. May.
\newblock The spectra associated to permutative categories.
\newblock {\em Topology}, 17(3):225--228, 1978.

\bibitem{book:MayClassifying}
J.~Peter May.
\newblock Classifying spaces and fibrations.
\newblock {\em Mem. Amer. Math. Soc.}, 1(1, 155):xiii+98, 1975.

\bibitem{paper:SegalMcDuff}
D.~McDuff and G.~Segal.
\newblock Homology fibrations and the ``group-completion'' theorem.
\newblock {\em Invent. Math.}, 31(3):279--284, 1975/76.

\bibitem{paper:Mingo}
J.~A. Mingo.
\newblock {$K$}-theory and multipliers of stable {$C^\ast$}-algebras.
\newblock {\em Trans. Amer. Math. Soc.}, 299(1):397--411, 1987.

\bibitem{paper:Nawata-Watatani}
Norio Nawata and Yasuo Watatani.
\newblock Fundamental group of simple {$C^*$}-algebras with unique trace.
\newblock {\em Adv. Math.}, 225(1):307--318, 2010.

\bibitem{paper:Nistor_AF}
Victor Nistor.
\newblock Fields of {${\rm AF}$}-algebras.
\newblock {\em J. Operator Theory}, 28(1):3--25, 1992.

\bibitem{Rieffel:Quantization}
Marc~A. Rieffel.
\newblock Quantization and {$C^\ast$}-algebras.
\newblock In {\em {$C^\ast$}-algebras: 1943--1993 ({S}an {A}ntonio, {TX},
  1993)}, volume 167 of {\em Contemp. Math.}, pages 66--97. Amer. Math. Soc.,
  Providence, RI, 1994.

\bibitem{Ror:encyclopedia}
M.~R{\o}rdam.
\newblock {\em Classification of nuclear, simple {$C\sp *$}-algebras}, volume
  126 of {\em Encyclopaedia Math. Sci.}
\newblock Springer, Berlin, 2002.

\bibitem{paper:Rordam-ZStable}
Mikael R{\o}rdam.
\newblock The stable and the real rank of {$\mathcal{Z}$}-absorbing
  {$C^*$}-algebras.
\newblock {\em Internat. J. Math.}, 15(10):1065--1084, 2004.

\bibitem{paper:RosenbergCT}
Jonathan Rosenberg.
\newblock Continuous-trace algebras from the bundle theoretic point of view.
\newblock {\em J. Austral. Math. Soc. Ser. A}, 47(3):368--381, 1989.

\bibitem{Schochet:dummy}
Claude Schochet.
\newblock The {D}ixmier-{D}ouady invariant for dummies.
\newblock {\em Notices Amer. Math. Soc.}, 56(7):809--816, 2009.

\bibitem{paper:Schoen}
Rolf Sch{\"o}n.
\newblock Fibrations over a {CW}h-base.
\newblock {\em Proc. Amer. Math. Soc.}, 62(1):165--166 (1977), 1976.

\bibitem{Segal:paper-Classifying-spaces}
Graeme Segal.
\newblock Classifying spaces and spectral sequences.
\newblock {\em Inst. Hautes \'Etudes Sci. Publ. Math.}, (34):105--112, 1968.

\bibitem{paper:SegalCatAndCoh}
Graeme Segal.
\newblock Categories and cohomology theories.
\newblock {\em Topology}, 13:293--312, 1974.

\bibitem{paper:SegalString}
Graeme Segal.
\newblock Topological structures in string theory.
\newblock {\em R. Soc. Lond. Philos. Trans. Ser. A Math. Phys. Eng. Sci.},
  359(1784):1389--1398, 2001.
\newblock Topological methods in the physical sciences (London, 2000).

\bibitem{paper:A.Strom}
Arne Str{\o}m.
\newblock Note on cofibrations.
\newblock {\em Math. Scand.}, 19:11--14, 1966.

\bibitem{paper:Thomason}
Robert~W. Thomason.
\newblock First quadrant spectral sequences in algebraic {$K$}-theory via
  homotopy colimits.
\newblock {\em Comm. Algebra}, 10(15):1589--1668, 1982.

\bibitem{Thomsen:homotopy-duke}
Klaus Thomsen.
\newblock Homotopy classes of {$*$}-homomorphisms between stable
  {$C^*$}-algebras and their multiplier algebras.
\newblock {\em Duke Math. J.}, 61(1):67--104, 1990.

\bibitem{paper:TomsWinter}
Andrew~S. Toms and Wilhelm Winter.
\newblock Strongly self-absorbing {$C^*$}-algebras.
\newblock {\em Trans. Amer. Math. Soc.}, 359(8):3999--4029, 2007.

\bibitem{paper:WinterZStable}
Wilhelm Winter.
\newblock Strongly self-absorbing {$C^*$}-algebras are {$\mathcal{Z}$}-stable.
\newblock {\em J. Noncommut. Geom.}, 5(2):253--264, 2011.

\end{thebibliography}

\end{document}